\documentclass[11pt]{amsart}
\usepackage{geometry}
\usepackage[english]{babel}
\usepackage[latin1]{inputenc}
\usepackage{amsrefs,amsfonts,amsmath,amssymb,amsthm,mathtools}

\usepackage{graphicx}
\usepackage{hyperref}
\usepackage[capitalize]{cleveref}
\usepackage{verbatim}
\usepackage{pdflscape}
\usepackage{tikz}
\usepackage{tikz-cd}

	\theoremstyle{plain} 
	\newtheorem{theorem}{Theorem}
	\newtheorem{lemma}{Lemma}
	\newtheorem{corollary}{Corollary}
	\newtheorem{proposition}{Proposition}

	\theoremstyle{definition}
	\newtheorem{definition}{Definition}
	\newtheorem{example}{Example}
	\newtheorem{remark}{Remark}

	\DeclareMathOperator{\N}{\mathbb N}
	\DeclareMathOperator{\R}{\mathbb R}
	\DeclareMathOperator{\C}{\mathbb C}
	
	\DeclareMathOperator{\re}{Re}
	\DeclareMathOperator{\im}{Im}

	\DeclareMathOperator{\Span}{span}
	
	\DeclareMathOperator{\rank}{rk}

	\DeclareMathOperator{\hol}{\mathfrak{hol}}
	\DeclareMathOperator{\aut}{\mathfrak{aut}}

	\newcommand{\Sphere}[1]{S^{#1}}

\begin{document}

\title[]{The reflection map and infinitesimal deformations of sphere mappings}
\author{Michael Reiter}
\address{Faculty of Mathematics, University of Vienna}
\email{m.reiter@univie.ac.at}

\begin{abstract}
The reflection map introduced by D'Angelo is applied to deduce simpler descriptions of nondegeneracy conditions for sphere maps and to the study of infinitesimal deformations of sphere maps. It is shown that the dimension of the space of infinitesimal deformations of a nondegenerate sphere map is bounded from above by the explicitly computed dimension of the space of infinitesimal deformations of the homogeneous sphere map. Moreover a characterization of the homogeneous sphere map in terms of infinitesimal deformations is provided.
\end{abstract}

\thanks{The author was supported by the Austrian Science Fund (FWF), project P28873-N35.}

\maketitle

\section{Introduction}

The main motivation is the study of real-analytic CR maps of the \textit{unit sphere} $\Sphere{2n-1}$ in $\C^n$ for $n\geq 2$, which is defined by
\begin{align*}
\Sphere{2n-1} = \{z=(z_1,\ldots, z_n)\in \C^n: \|z\|^2 = |z_1|^2 + \ldots + |z_n|^2 = 1\}.
\end{align*}
For $n=2$ write $z = z_1, w = z_2$. A lot is known about mappings of spheres, see the survey by D'Angelo \cite{DAngelo15} and the references therein. 
A prominent example of a sphere map is the \textit{homogeneous sphere map $H_n^d$ of degree $d$} from $\Sphere{2n-1}$ into $\Sphere{2K-1}$ for some $K=K(n,d)\in \N$, which consists of all lexicographically ordered monomials in $z=(z_1,\ldots,z_n)\in \C^n$ of degree $d$ and is given by 
\begin{align*}
H^d_n(z) = \left( \sqrt{\binom{\alpha}{d}} z^\alpha \right)_{|\alpha|=d}.
\end{align*}

The purpose of this article is to study the \textit{reflection map}, which was introduced by D'Angelo \cite{DAngelo03} in the case of sphere mappings and further investigated by the same author in \cite{DAngelo07a} in the case of maps of hyperquadrics. The reflection map of a mapping $H$ allows to effectively compute and deduce several properties of the \textit{X-variety} associated to $H$. The X-variety was introduced and studied by Forstneri\v{c} in \cite{Forstneric89} to extend CR maps satisfying certain smoothness assumptions. In the case of real-analytic CR maps of spheres it is shown that these maps are rational.

The homogeneous sphere map $H_n^d$ plays a crucial role in the classification of polynomial maps, see the works of D'Angelo \cites{DAngelo88b, DAngelo91} and \cite{DAngelo16} for rational sphere maps. The homogeneous sphere map appears in the definition of the reflection map $C_H$ for a rational sphere map $H=P/Q: \Sphere{2n-1} \rightarrow \Sphere{2m-1}$ with $Q\neq 0$ on $\Sphere{2n-1}$: Let $V_H: \C^m \rightarrow \C^K$ be a matrix with holomorphic entries, satisfying $V_H(X) \cdot \bar H_n^d /\bar Q = X \cdot \bar H$ on $\Sphere{2n-1}$ for $X\in \C^m$, where $\cdot$ denotes the euclidean inner product. The previous identity is achieved by the homogenization technique of D'Angelo \cite{DAngelo88b}. $V_H$ is referred to as \textit{reflection matrix} and $C_H(X) \coloneqq V_H(X) \cdot \bar H_n^d/\bar Q$ for $X\in \C^m$. See \cref{sec:refMatrix} below for more details.

In this article the reflection matrix will be applied in two ways. In the first case it is shown that nondegeneracy conditions of sphere maps can be rephrased in terms of rank conditions on the reflection matrix.
The nondegeneracy conditions considered here were introduced in \cite{Lamel01} and \cite{LM17} respectively. In the case of a sphere map $H: \Sphere{2n-1} \rightarrow \Sphere{2m-1}$ they are defined as follows: If $\Gamma$ denotes the set of real-analytic CR vector fields tangent to $\Sphere{2n-1}$, then $H$ is called \textit{finitely nondegenerate} at $p\in \Sphere{2n-1}$, if there is an integer $\ell \in \N$, such that,
\begin{align*}
\Span_{\C}\{L_1 \cdots L_k H(z)|_{z=p}: L_j \in \Gamma, k\leq \ell\} = \C^m.
\end{align*}
The map $H$ is called \textit{holomorphically nondegenerate} if there is no nontrivial holomorphic vector field tangent to $\Sphere{2m-1}$ along the image of $H$. 
These notions of nondegeneracy were originally defined for submanifolds and introduced by \cite{Stanton95} and \cite{BHR96} respectively, see also the survey of Lamel \cite{Lamel11}.
For more details on nondegeneracy conditions for CR maps see also \cref{sec:nonDegMaps} below. Then the following theorem is shown:

\begin{theorem}
\label{thm:rankVNonDegIntro}
Let $H: \Sphere{2n-1} \rightarrow \Sphere{2m-1}$ be a rational map of degree $d$.
\begin{itemize}
\item[(a)] $H$ is finitely nondegenerate at $p \in \Sphere{2n-1}$ if and only if $V_H$ is of rank $m$ at $p\in\Sphere{2n-1}$.
\item[(b)] $H$ is holomorphically nondegenerate if and only if $V_H$ is generically of rank $m$ on $\Sphere{2n-1}$.
\end{itemize}
\end{theorem}

This has immediate consequences to show sufficient and necessary conditions in terms of nondegeneracy conditions for the X-variety of $H$ to be an affine bundle or that it agrees with the graph of the map, see \cref{sec:XVariety} and \cref{thm:XVariety} below for more details.

In the second case, applications of the reflection matrix to the study of infinitesimal deformations are provided. For $M\subset \C^N$ and $M'\subset \C^{N'}$ real submanifolds consider the set $\mathcal H(M,M')$ of all maps, which are holomorphic in a neighborhood of $M$ and satisfying $H(M) \subset M'$. In \cites{dSLR15a,dSLR15b,dSLR17,dSLR18} \textit{locally rigid} maps were studied. They correspond to isolated points in the quotient space of $\mathcal H(M,M')$ under automorphisms. A sufficient linear condition was provided for local rigidity of a given map, which is formulated in terms of \textit{infinitesimal deformations}. An infinitesimal deformation of a map $H:M \rightarrow M'$ is a holomorphic vector, defined in a neighborhood of $M$, whose real part is tangent to $M'$ along the image of $H$. The set of infinitesimal deformations of a map $H$ is denoted by $\hol(H)$. 
Examples of infinitesimal deformations of a map $H$ can be obtained from smooth curves of maps $\R \ni t \mapsto H(t)$, with $H(0) = H$, since $\frac{d H(t)}{dt}|_{t=0} \in \hol(H)$.

The results involving infinitesimal deformations are summarized in the following theorem:

\begin{theorem}
\label{thm:polyMapInfDefIntro}
Let $H: \Sphere{2n-1} \rightarrow \Sphere{2m-1}$ be a holomorphically nondegenerate rational map of degree $d$.
It holds that $\dim \hol(H) \leq \dim \hol(H^d_n)=\left(\frac{2d+n}{d}\right)K(n,d)^2$ and if $H$ is assumed to be polynomial, then $\dim \hol(H) = \dim \hol(H^d_n)$ if and only if $H$ is unitarily equivalent to $H^d_n$.
\end{theorem}

This result contains an alternative characterization of the homogeneous sphere map to the one given in \cites{Rudin84, DAngelo88b} or \cite{DAngeloBook}*{section 5.1.4, Theorem 3} and demonstrates a new method to compute infinitesimal deformations for sphere maps. While the article \cite{dSLR15a} contains examples which required computer-assistance, it is shown in several examples in this article that the reflection matrix allows for explicit and effective computations of infinitesimal deformations of sphere maps.

\section{Preliminaries}


The purpose of this section is to introduce the necessary notions and notations needed throughout the article. These are only required for maps of spheres but without any effort and no loss of clarity the general case of maps of manifolds $M$ and $M'$ is treated. To this end the following assumptions are made: Let $M$ be a real-analytic generic submanifold of $\C^N$ of codimension $d$. For a real-analytic CR submanifold $M'\subset \C^{N'}$ of codimension $d'$, let $p'\in M'$ and $\rho': V' \times \bar V' \rightarrow \R^{d'}, \rho'=(\rho_1',\ldots, \rho_{d'}')$, be a real-analytic mapping, such that $M'\cap V'=\{z'\in V' : \rho'(z',\bar z') = 0\}$, where $V' \subset \C^{N'}$ is a neighborhood of $p' \in M'$ and the differentials $d\rho'_1,\ldots, d\rho'_{d'}$ are linearly independent in $V'$.
Denote $\bar V' = \{\bar z'\in \C^{N'}: z' \in V'\}$. The complex gradient ${\rho'_j}_{z'}$ of $\rho'_j$ is given by ${\rho'_j}_{z'} = \left(\frac{\partial \rho'_j}{\partial z_1'}, \ldots,\frac{\partial \rho'_j}{\partial z_{N'}'}\right)$. The following notation is used: $v \cdot w \coloneqq v_1 w_1 + \cdots + v_n w_n$ for vectors $v=(v_1,\ldots, v_n) \in \C^n$ and $w =(w_1,\ldots, w_n) \in \C^n$.

\subsection{Infinitesimal deformations of CR maps}

One of the main objects of this article are \textit{infinitesimal deformations} of a CR map.

\begin{definition}
Let $H: M \rightarrow M'$ be a real-analytic CR map. 
A real-analytic CR map $X: M \rightarrow \C^{N'}$ is called an \textit{infinitesimal deformation of $H$}, if for every $p\in M$ and every real-analytic mapping $\rho'=(\rho'_1,\ldots, \rho'_{d'})$ defined in a neighborhood of $H(p)$ vanishing on $M'$, it holds that,
\begin{align*}
\re(X(z) \cdot {\rho'_j}_{z'}(H(z),\overline{H(z)})) = 0, \quad z\in M \cap U, \quad j=1,\ldots, d',
\end{align*}
for some open neighborhood $U \subset \C^N$ of $p$.
The space of infinitesimal deformations of $H$ is denoted by $\hol(H)$. 

For a real manifold $M$ the space $\hol(M)$ of \textit{infinitesimal automorphisms of $M$} consists of holomorphic vectors whose real part is tangent to $M$. 

For a map $H: M \rightarrow M'$ the subspace $\aut(H) \coloneqq \hol(M')|_{H(M)} + H_*(\hol(M)) \subset \hol(H)$ is referred to as the space of \textit{trivial infinitesimal deformations} of $H$. Its complement in $\hol(H)$ is called the space of \textit{nontrivial infinitesimal deformations} of $H$.

A map $H$ is called \textit{infinitesimally rigid} if $\hol(H) = \aut(H)$.

The \textit{infinitesimal stabilizer} of $H$ is given by $(S,S') \in \hol(M) \times \hol(M')$ such that $H_*(S) = - S'|_{H(M)}$. An infinitesimal automorphism $S \in \hol (M)$ is said to belong to the infinitesimal stabilizer of $H$ if there exists $S' \in \hol(M')$ such that $H_*(S) = - S'|_{H(M)}$.
\end{definition}

In the case of sphere mappings, for a real-analytic CR map $H: \Sphere{2k-1} \rightarrow \Sphere{2m-1}$, 
a holomorphic map $X: U \rightarrow \C^n$, where $U \subset \C^k$ is an open neighborhood of $\Sphere{2k-1}$, is an infinitesimal deformation of $H$ if $\re(X(z) \cdot \overline{H(z)}) = 0$ for $z\in\Sphere{2k-1}$.

\subsection{Nondegeneracy conditions for CR maps}
\label{sec:nonDegMaps}

The purpose of this section is to provide the definitions of finite and holomorphic nondegeneracy for CR maps introduced by Lamel \cite{Lamel01} and Lamel--Mir \cite{LM17} respectively, and study some of their properties.

\begin{definition}
A real-analytic CR map $H: M \rightarrow M'$ is called \textit{holomorphically degenerate} if there exists a real-analytic CR map $Y: M \rightarrow \C^{N'}$ satisfying $Y\not\equiv 0$ and for every $p\in M$ and every real-analytic mapping $\rho'=(\rho'_1,\ldots, \rho'_{d'})$ defined in a neighborhood of $H(p)$ vanishing on $M'$, it holds that, 
\begin{align*}
Y(z) \cdot {\rho'_j}_{z'}(H(z),\overline{H(z)}) = 0, \quad z \in M \cap U,\quad j=1,\ldots, d',
\end{align*}
for some open neighborhood $U \subset \C^N$ of $p$.
If a map is not holomorphically degenerate it is called \textit{holomorphically nondegenerate}.
\end{definition}

Simple examples of holomorphically degenerate maps are the following:

\begin{example}
Let $F: \Sphere{2n-1} \rightarrow \Sphere{2m-1}$, then $H = F \oplus 0$, where $0 \in \C^k$, is a holomorphically degenerate sphere map from $\Sphere{2n-1}$ into $\Sphere{2(n+k)-1}$, since $X=0 \oplus G$ for $0\in \C^n$ and $G$ a holomorphic function from $\C^n$ into $\C^k$ satisfies $X \cdot \bar H = 0$.
\end{example}

\textit{Finite nondegeneracy} is defined as follows:

\begin{definition}
\label{def:finiteDeg}
Consider  a real-analytic CR map $H: M \rightarrow M'$. Let $\bar L_1, \ldots, \bar L_n$ a basis of CR vector fields of $M$ and for a multiindex $\alpha=(\alpha_1,\ldots, \alpha_n)\in \N^n$ denote $\bar L^{\alpha} = \bar L_1^{\alpha_1} \cdots \bar L_n^{\alpha_n}$. Let $p\in M$. For each $k\in \N$ define the following subspaces of $\C^{N'}$:
\begin{align*}
E_k'(p) = \Span_{\C} \{\bar L^\alpha {\rho'_j}_{z'}(H(z),\overline{H(z)})|_{z=p}: |\alpha| \leq k,1\leq j \leq d'\},
\end{align*}
for a real-analytic mapping $\rho'=(\rho'_1,\ldots, \rho'_{d'})$ defined in a neighborhood of $H(p)$ and vanishing on $M'$.
Define $s(p) \coloneqq N' - \max_k \dim_{\C} E_k'(p)$, which is called the \textit{degeneracy} of $H$ at $p$. The map $H$ is of \textit{constant degeneracy} $s(q)$ at $q\in M$ if $p\mapsto s(p)$ is a constant function in a neighborhood of $q$. If $s(p)=0$, then $H$ is called \textit{finitely nondegenerate} at $p$. Considering the smallest integer $k_0$ such that $E_\ell'(p) = E_{k_0}'(p)$ for all $\ell \geq k_0$ one can say more precisely that $H$ is \textit{$(k_0,s)$-degenerate} at $p$. If the map is finitely nondegenerate at $p$ one says that it is \textit{$k_0$-nondegenerate} at $p$.
\end{definition}

Note that if a map is finitely nondegenerate at $p$, then it is also finitely nondegenerate at points in a neighborhood of $p$. If $M$ is connected, by \cite{Lamel01}*{Lemma 22} the set of points where the map $H:M \rightarrow M'$ is of constant degeneracy $s(H)\coloneqq \min_{p\in M}s(p)$ is an open and dense subset of $M$. The number $s(H)$ is called \textit{generic degeneracy of $H$}. 


Constant degeneracy can be phrased in terms of vector fields as follows:

\begin{proposition}[Lamel \cite{Lamel01}*{Proposition 18}]
\label{prop:finiteNonDegVF}
Let $H: M \rightarrow M'$ be a real-analytic CR mapping of constant degeneracy $s$ in a neighborhood of $p\in M$. Then it holds that 
\begin{align*}
s= \dim_{\C}\{X(p): X \text{ is a real-analytic CR vector field tangent to } M' \text{ along } H(M)\}.
\end{align*}
\end{proposition}

These following statements are analogous to the corresponding statements for manifolds or infinitesimal automorphisms, see \cite{BERbook}*{Theorem 11.5.1} and \cite{BERbook}*{Proposition 12.5.1}. 

\begin{proposition}
\label{prop:propNondeg}
Let $H:M \rightarrow M'$ be a real-analytic CR map. Then the following statements hold:
\begin{itemize}
\item[(a)] If $H$ is holomorphically degenerate, then $\dim_{\R} \hol(H) = \infty$. 
\item[(b)] If $H$ is finitely nondegenerate at $p\in M$, then it is holomorphically nondegenerate.
\item[(c)] If the space of holomorphic vector fields at $p\in M$ tangent to $M'$ along the image of $H$ is of complex dimension $s$, then, outside a proper real-analytic variety of a neighborhood of $p$, the map $H$ is of constant degeneracy $s$.
\end{itemize}
\end{proposition}

Note that if $s=0$, then (c) says that if the map $H$ is holomorphically nondegenerate then $H$ is finitely nondegenerate outside a proper real-analytic variety.

Moreover, if $M$ is assumed to be connected, (c) yields the following statement: If at any $p\in M$ the space of holomorphic vector fields at $p\in M$ tangent to $M'$ along the image of $H$ is of complex dimension at least $s$, then the generic degeneracy of $H$ is at least $s$. 

\begin{proof}[Proof of \cref{prop:propNondeg}]
To show (a) argue as in \cite{BERbook}*{Proposition 12.5.1}: If $H$ is holomorphically degenerate, there exists a nontrivial holomorphic map $X$ tangent to $M'$ along $H(M)$. Then for each $k\in \N$ also $Y_k \coloneqq z_1^k X$ is tangent to $M'$ along $H(M)$ and these maps are complex-linearly independent. 
Since $M$ is generic and the real part of a nontrivial holomorphic map $\hat X$, restricted to $M$, cannot vanish on $M$ (the vanishing of $\re(\hat X)|_M$ would imply that $\hat X \equiv 0$), the vector fields $\re(Y_k)$ are real-linearly independent, hence $\dim_{\R} \hol(H) = \infty$. 

To prove (b) denote by $\mathfrak X(H)$ the set of holomorphic vector fields tangent to $M'$ along the image of $H$. 
Consider $X = \sum_j a_j(Z) \frac{\partial}{\partial Z_j} \in \mathfrak X(H)$ which, by the finite  nondegeneracy of $H$ and \cref{prop:finiteNonDegVF}, satisfies $a_j(p)=0$ and $\sum_j a_j(Z) {\rho'_k}_{Z_j}(H(Z),\overline{H(Z)}) = 0$ for $Z\in M, k=1,\ldots, d'$ and a real-analytic mapping $\rho'=(\rho'_1,\ldots, \rho'_{d'})$ defined in a neighborhood of $H(p)$ and vanishing on $M'$. Taking derivatives w.r.t. CR vector fields $L$ of $M$ one gets
\begin{align}
\label{eq:finiteHolNondeg}
\sum_j a_j(Z) L^{\beta_m} {\rho'_k}_{Z_j}(H(Z),\overline{H(Z)}) = 0, \quad Z \in M,\quad 1\leq k \leq d',
\end{align}
for multiindices $\beta_m \in \N^n$ for $m = 1,\ldots,N'$. Use coordinates as given in e.g.  \cite{BERbook}*{Proposition 1.3.6} for $M$ in \eqref{eq:finiteHolNondeg}, i.e.  $Z=(z,u+i\phi(z,\bar z,u)) \in \C^n \times \C^d$, where $\phi$ is defined near $0$ in $\R^{2N+d}$ with values in $\R^d$ satisfying $\phi(p)=0$ and $d\phi(p)=0$. Taking derivatives w.r.t. $z$ and $u$ one gets:
\begin{align*}
\sum_j {a_j}_{Z_r}(Z) L^{\beta_m} {\rho'_k}_{Z_j}(H(Z),\overline{H(Z)}) + \text{l.o.t.} = 0, \quad 1\leq r \leq N, 1\leq k \leq d',
\end{align*}
where the expression ``l.o.t.'' stands for terms vanishing at $Z=p$. Evaluating at $Z=p$ one gets that $\sum_j {a_j}_{Z_r}(p) L^{\beta_m} {\rho'_k}_{Z_j}(H(Z),\overline{H(Z)})|_{Z=p} = 0$ for $1\leq k \leq d'$. Thus, since $H$ is finitely nondegenerate at $p$, there are multiindices $\gamma_m\in \N^n$ and integers $k_m \in \N$ with $1\leq k_m\leq d'$, such that the matrix $(L^{\gamma_m} {\rho'_{k_m}}_{Z_j}(H(Z),\overline{H(Z)})|_{Z=p})_{1\leq j,m\leq N'}$ is of full rank, which implies that ${a_j }_{Z_\ell}(p) = 0$. Proceeding inductively shows that  all derivatives of $a_j(Z)$ have to vanish at $p$. This means that the holomorphic vector field $X$ vanishes in a neighborhood of $p$ on the generic submanifold $M$, hence $X \equiv 0$, which implies that $H$ is holomorphically nondegenerate.

To show (c) let $\{X_1,\ldots, X_s\}$ be a basis of the space of holomorphic vector fields at $p\in M$ tangent to $M'$ along the image of $H$ and take $X\in \Span_{\C}\{X_1,\ldots, X_s\}$. Consider the following equation for $1\leq k \leq d'$:
\begin{align*}
X(Z) \cdot {\rho'_k}_{Z}(H(Z),\overline{H(Z)})=0, \quad Z\in M.
\end{align*}
Take derivatives w.r.t. $L$ and since $X$ is holomorphic, it holds that,
\begin{align*}
X(Z) \cdot L^{\beta} {\rho'_k}_{Z}(H(Z),\overline{H(Z)}) = 0, \quad Z\in M,
\end{align*}
for any multiindex $\beta\in \N^n$. This implies that for any sequence $\beta_1,\ldots, \beta_{N'}\in \N^n$ of multiindices and integers $k_m \in \N$ with $1\leq k_m\leq d'$ the vector field $X$ belongs to the kernel of the matrix $(L^{\beta_m} {\rho'_{k_m}}_{Z}(H(Z),\overline{H(Z)}))_{1\leq m \leq N'}$. Hence outside a proper real-analytic variety $Y$ of a neighborhood of $p$ it holds that for any $\ell\in \N$ one has $\dim_{\C} E_\ell'(q) = N'-s$ for $q \in Y$, such that the degeneracy of $H$ is equal to $s$ for all $q\in Y$.
\end{proof}

\subsection{Infinitesimal automorphisms of the unit sphere}
\label{sec:generalInfAuto}
In the following the well-known infinitesimal automorphisms of $\Sphere{2n-1}, n \geq 2$ are listed for later reference. 
For $A=(A_1, \ldots, A_n) \in \hol(\Sphere{2n-1})$ the $j$-th component is given as follows:
\begin{align*}
A_j = \alpha_j + \sum_{1\leq \ell \leq j-1} \beta_j^\ell z_\ell + i s_j z_j - \sum_{j+1\leq \ell \leq n} \bar \beta_\ell^j z_\ell - \sum_{1\leq \ell \leq n} \bar \alpha_\ell z_j z_\ell,
\end{align*}
where $\alpha_m, \beta_m^{\ell} \in \C$ and $s_m \in \R$ and $\dim_{\R} \hol(\Sphere{2n-1})= n(n+2)$. 
The following notation is required:
\begin{align*}
S_n^1 & =  \Bigl(\alpha_j - \sum_{1\leq \ell \leq n} \bar \alpha_\ell z_j z_\ell\Bigr)_{1\leq j \leq n}, & \qquad 
S_n^2 & = \Bigl(\sum_{1\leq \ell \leq j-1} \beta_j^\ell z_\ell - \sum_{j+1\leq \ell \leq n} \bar \beta_\ell^j z_\ell \Bigr)_{1\leq j \leq n},\\
S_n^3 & =  i (s_j z_j)_{1\leq j \leq n}.
\end{align*}

For a map $H: \Sphere{2k-1} \rightarrow \Sphere{2m-1}$ any $T \in \aut(H)$ can be written as $T = T_1 + \ldots + T_4$, such that the $T_j$ are given as follows:
\begin{align*}
T_1 & =\alpha' - (\bar \alpha' \cdot H) H,
& T_2  &  = V' H, \\
T_3  & = H_*  \alpha -(\bar \alpha \cdot  z) H_* z,
& T_4  &  = H_* (S_k^2+S_k^3),
\end{align*}
where $\alpha'  = (\alpha'_1,\ldots, \alpha'_{m}) \in \C^{m}$, $V'\in\C^{m}\times \C^{m}$ is a matrix satisfying $V' + \prescript{t}{}{\bar V'} = 0$, $\alpha = (\alpha_1,\ldots, \alpha_k) \in \C^k$ and $z=(z_1,\ldots,z_k) \in \C^k$.

\subsection{The reflection matrix}
\label{sec:refMatrix}
The following definition is a summary of \cite{DAngelo03}*{Definition 2.1, 2.2} introducing the homogenization and reflection map (which appears in the study of the $X$-variety, see also \cite{DAngelo07a} for the case of hyperquadric maps): Denote by $\mathcal H(n,d)$ the complex vector space of homogeneous polynomials of degree $d$ in $n$ holomorphic variables $z = (z_1, \ldots, z_n)$. Write $\bar {\mathcal H}(n,d)$ for the complex vector space with basis consisting of homogeneous polynomials of degree $d$ in $n$ anti-holomorphic variables $\bar z = (\bar z_1, \ldots, \bar z_n)$.

\begin{definition}
\label{def:VH}
Let $H = \frac P Q: U \subset \C^n \rightarrow \C^m$ be a rational map of degree $d$ (not necessarily a sphere map), where $P=(P_1, \ldots, P_m)$ and $Q: \C^n \rightarrow \C$ with $Q \neq 0$ on $U$. Write $H = \frac 1 Q \sum_{k=0}^d P^k$, where $P^k$ is homogeneous of order $k$. 
Define the \textit{reflection map} $C_H: \C^m \rightarrow \bar {\mathcal H}(n,d)$ of $H$ by
\begin{align}
\label{eq:DefVH}
C_H(X) \coloneqq  X \cdot \frac 1 {\bar Q} \sum_{k=0}^d \bar P^k(\bar z) \|z\|^{2(d-k)}, \quad X \in \C^m. 
\end{align} 
Since $C_H$ is linear and $C_H(X) \in \bar {\mathcal H}(n,d)$ there exists a matrix $V_H: \C^m \rightarrow \C^{K(n,d)}$ with holomorphic entries, such that 
 \begin{align*}
C_H(X) =  V_H X \cdot \frac{\bar H^d_n}{\bar Q}, \quad X \in \C^m,
\end{align*}
and denote $V \coloneqq Q V_H$. The matrix $V$ is referred to as \textit{reflection matrix of $H$}. 
\end{definition}

Several properties of the reflection matrix and examples involving $V$ are given in \cite{DAngelo03}. 

\section{Examples and constructions for sphere maps}
\label{sec:NonDegSphereMaps}

In this section some particular examples of sphere maps and constructions of sphere maps are presented and their relation to the above nondegeneracy conditions are discussed.

\subsection{The homogeneous sphere maps}
\label{sec:homMap}

The purpose of this section is to show some properties of the homogeneous sphere maps defined as follows:

\begin{definition}
\label{def:homSphereMap}
For $d\geq 1$ and $n\geq 2$ define $K(n,d)=\binom{n+d-1}{d}$ and $I(n,d)$ as the set of all multiindices $\alpha \in \N^n$ of length $d$ equipped with the lexicographic order. Define the \textit{homogeneous sphere map $H_n^d$ in $n$ variables of degree $d$} as
\begin{align*}
H^d_n: \Sphere{2n-1} \rightarrow \Sphere{2K(n,d)-1}, \quad H^d_n(z) = \left( a^{dn}_\alpha z^\alpha \right)_{\alpha\in I(n,d)}, \quad (a^{dn}_\alpha)^2 = \binom{\alpha}{d}.
\end{align*}
\end{definition}

A direct computation or \cite{DX17}*{Theorem 4.2} shows that the infinitesimal stabilizer of $H^d_n$ is given by $S_n^2$ and $S_n^3$.
One can show that $H^d_n$ is holomorphically nondegenerate by using the Fourier coefficient technique as in \cite{DAngelo88b}*{Lemma 16}. Instead of showing this, it is proved that $H_n^d$ is finitely nondegenerate on $\Sphere{2n-1}$. Before giving a proof of this fact some preparations are needed:

For $n\geq 3$ we define CR vector fields of $\Sphere{2n-1}$  by $\bar L_{ij} = z_i \frac{\partial}{\partial \bar z_j} - z_j \frac{\partial}{\partial \bar z_i}$ for $1 \leq i \neq j \leq n$. For $n=2$ the CR vector field of $\Sphere{3}$ is given by $\bar L = z\frac{\partial}{\partial \bar w}-w\frac{\partial}{\partial \bar z}$.

Let $\{X_{ij}: 1\leq i,j\leq n\}$ be a collection of vector fields. In order to denote powers of such vector fields the following notation is used: Define the set $J \coloneqq \{\alpha =(\alpha_1,\ldots, \alpha_n)\in \N^{3n}: \alpha_j=(\alpha_j^1,\alpha_j^2,\alpha_j^3) \in \N^3\}$ and for $\alpha \in J$ write $X^\alpha \coloneqq X_{\alpha_1^1\alpha_1^2}^{\alpha_1^3} \cdots X_{\alpha_n^1\alpha_n^2}^{\alpha_n^3}$. Define $|\alpha| = \sum_{j=1}^n \alpha_j^3$ for $n\geq 3$.

For two vector fields $X$ and $Y$, their Lie bracket is denoted by $[X,Y] \coloneqq X(Y)-Y(X)$.

In the following lemma some basic facts about CR vector fields and their commutators are given. The proofs consist of straight forward calculations and are omitted.

\begin{lemma}
\label{lem:commutator}
Assume $n\geq 3$. In the following for $1\leq i,j,k,\ell \leq n$ assume $i\neq j$ and $k,\ell\not\in\{i,j\}$. Define the following vector fields for $\Sphere{2n-1}$:
\begin{align*}
 T_{kj} & \coloneqq  [L_{ij},\bar L_{ik}] = -\bar z_k \frac{\partial}{\partial \bar z_j} + z_j \frac{\partial}{\partial z_k},\\
 S_{ij} & \coloneqq  [L_{ij},\bar L_{ij}] = - z_i \frac{\partial}{\partial z_i} - z_j \frac{\partial}{\partial z_j}+ \bar z_i \frac{\partial}{\partial \bar z_i} + \bar z_j \frac{\partial}{\partial \bar z_j}.
\end{align*}
Then $T_{jk} = - \bar T_{kj}$ and $S_{ij} = S_{ji}$ and the following commutator relations hold: 
\begin{itemize}
\item[(a)] $[T_{k\ell},\bar L_{ij}] = [T_{kj},\bar L_{ij}] = [T_{ki},\bar L_{ij}] = [T_{ij},\bar L_{ij}] = [T_{ji},\bar L_{ij}] = 0$, $[T_{jk},\bar L_{ij}]= \bar L_{ik}, [T_{ki},\bar L_{kj}] = \bar L_{i j}$. 
\item[(b)] $[S_{k\ell},\bar L_{ij}] = 0$, $[S_{kj},\bar L_{ij}] =[S_{ki},\bar L_{ij}]  = - \bar L_{ij}, [S_{ij},\bar L_{ij}] = - 2 \bar L_{ij}$.
\item[(c)] $[T_{k\ell}, L_{ij}] = [T_{jk}, L_{ij}] = [T_{ki},\bar L_{kj}] = [T_{ij}, L_{ij}]  = [T_{ji}, L_{ij}] = 0$, $[T_{kj}, L_{ij}] = L_{ki}, [T_{ki}, L_{ij}] = L_{jk}$.
\item[(d)] $[S_{k\ell},L_{ij}] = 0$, $[S_{kj},L_{ij}] =[S_{ki},L_{ij}]  = L_{ij}$, $[S_{ij}, L_{ij}] = 2 L_{ij}$.
\item[(e)] $[S_{ij},T_{k\ell}] = [S_{ij},T_{ij}] = 0$, $[S_{ij},T_{j\ell}] = T_{j\ell}$, $[S_{ij},T_{i\ell}] = T_{i\ell}$.
\end{itemize}
For $n=2$ define the following vector field for $\Sphere{3}$:
\begin{align*}
S\coloneqq [L,\bar L]= - \bar z \frac{\partial}{\partial \bar z} - \bar w \frac{\partial}{\partial \bar w} +z \frac{\partial}{\partial z} + w \frac{\partial}{\partial w}.
\end{align*}
Then it holds that $[S,\bar L] = 2 \bar L$.
\end{lemma}

\begin{lemma}
\label{lem:homMapNonDeg}
The map $H^d_n: \Sphere{2n-1} \rightarrow \Sphere{2K(n,d)-1}$ is $d$-nondegenerate at each point of $\Sphere{2n-1}$.
\end{lemma}

\begin{proof}
Set $H\coloneqq H^d_n$, fix $1\leq m \leq n$ and define the following set of multiindices 
\begin{align*}
J_m\coloneqq\{\alpha=(\alpha_1,\ldots, \alpha_n)\in J: \alpha_j = (m,j,k_j) \in \N^3, 1\leq  j \leq n, k_m=0\},
\end{align*}
where the index set $J$ from the beginning of \cref{sec:homMap} is used.
It will be shown that the $K(n,d) \times K(n,d)$-matrix 
\begin{align*}
A_m \coloneqq \left(L^\alpha H: \alpha\in J_m, 0\leq |\alpha|\leq d \right),
\end{align*}
is of full rank if $z_m \neq 0$. This implies that $H$ is $d$-nondegenerate at each point of  $\Sphere{2n-1}$. 
The proof consists of two steps:
\begin{itemize}
\item[\textbf{(A)}] It is proved that
\begin{align}
\label{eq:orthoDeriv}
L^\alpha H \cdot \bar L^\beta T^\gamma S^\delta \bar H = 0,
\end{align}
on $\Sphere{2n-1}$, for all multiindices $\alpha,\beta \in J_m$ with $|\alpha|<|\beta|$ and all $\gamma\in J,\delta\in J_m$, where $J$ is defined in the beginning of \cref{sec:homMap}.
\item[\textbf{(B)}] It is shown that for each $0\leq k \leq d$ the set $D_m^k\coloneqq \{L^\alpha H|_{\Sphere{2n-1}}: \alpha\in J_m, |\alpha|=k\}$ consists of linearly independent vectors in $\C^{K(n,d)}$ if $z_m\neq 0$.
\end{itemize}
Observe that the number of elements in $D_m^k$ is equal to $K(n-1,k)$, such that $\sum_{k=0}^d K(n-1,k) = K(n,d)$.
Thus, the two steps together imply that the matrix $A_m$ consists of linearly independent rows if $z_m\neq 0$. 
 
In order to show (A) one proceeds by induction on the length of $\alpha$ in \eqref{eq:orthoDeriv}: For $|\alpha|=0$ one needs to argue as follows. Since $H \cdot \bar H = 1$ on $\Sphere{2n-1}$, it follows that $p_{H,\beta}\coloneqq H \cdot \bar L^\beta \bar H = 0$ on $\Sphere{2n-1}$ for all $\beta \in J_m$. This means that $p_{H,\beta}$ is a homogeneous polynomial vanishing on $\Sphere{2n-1}$, hence $p_{H,\beta}$ vanishes in $\C^n$, see \cite{DAngelo91}*{section II} or \cite{DAngeloBook}*{section 5.1.4}. Applying $\bar z_k\frac{\partial}{\bar z_j}$ to $p_{H,\beta} \equiv 0$, implies that $H \cdot \bar L^\beta T^\gamma S^\delta \bar H= 0$ for all $\gamma\in J$ and $\delta \in J_m$.

Assume that \eqref{eq:orthoDeriv} holds for $|\alpha|=k$ and $|\alpha|+1 < |\beta|$. If one applies $L_{mj}$ to \eqref{eq:orthoDeriv} one obtains: 
\begin{align*}
L_{mj} L^\alpha H \cdot \bar L^\beta T^\gamma S^\delta \bar H + L^\alpha H \cdot L_{mj}\bar L^\beta T^\gamma S^\delta \bar H =0,
\end{align*}
on $\Sphere{2n-1}$. If one can show that
\begin{align}
\label{eq:indStep}
L^\alpha H \cdot L_{mj}\bar L^\beta T^\gamma S^\delta \bar H =0,
\end{align}
the induction is completed and (A) is proved. In order to show \eqref{eq:indStep}, use the identities from \cref{lem:commutator}, which imply that the expression of the left-hand side of \eqref{eq:indStep} can be rewritten as a sum of terms of the form $L^\alpha H \cdot \bar L^{\beta'} T^{\gamma'} S^{\delta'} L_{mj} \bar H$ and $L^\alpha H \cdot \bar L^{\beta'} T^{\gamma'} S^{\delta'} \bar H$, where $|\beta'|\geq |\beta|-1, \gamma'\in J$ and $\delta'\in J_m$. Hence using the induction hypothesis proves \eqref{eq:indStep}.

To prove (B) fix $0\leq k \leq d$ and assume that the set $D_m^k$ consists of vectors which are not linearly independent. By setting $K\coloneqq \{\alpha \in J_m: |\alpha|=k\}$, the linear dependence says that there are $c_\alpha\in \C$, not all of them are zero, such that
\begin{align}
\label{eq:linDepHom}
\sum_{\alpha\in K} c_\alpha L^\alpha H = 0,
\end{align}
on $\Sphere{2n-1}$. Since the left-hand side is a homogeneous polynomial, \eqref{eq:linDepHom} holds in $\C^n$. For $\alpha \in K$ define the multiindex
\begin{align*}
r(\alpha) \coloneqq \left((1,k_1),\ldots,(m-1,k_{m-1}),(m+1,k_{m+1}),\ldots,(n,k_n)\right)\in\N^{2(n-1)},
\end{align*}
such that $z^{r(\alpha)}\coloneqq z_1^{k_1} \dots z_{m-1}^{k_{m-1}} z_{m+1}^{k_{m+1}} \cdots z_n^{k_n}$. Note that $L^\alpha = \bar z_m^k \frac{\partial^k}{\partial z^{r(\alpha)}} + \ldots$, where $\ldots$ stands for derivatives of order $k$ with coefficients being monomials of degree $k$ containing $\bar z_m^\ell$ for $\ell < k$. Define $R_\alpha \coloneqq\frac{\partial^k}{\partial \bar z_m^k} L^\alpha$. 
Assume $z_m\neq 0$. Expanding \eqref{eq:linDepHom} as a power series in $\bar z$ with vector-valued coefficients, the coefficient of $\bar z_m^k$ is given by 
\begin{align}
\label{eq:linDepHom2}
\sum_{\alpha_K} c_\alpha  R_\alpha H = 0.
\end{align}
Note that in the vector $R_\alpha H\in\C^{K(n,d)}$ each monomial of $H^{d-k}$ appears in exactly one component and it is of the following form: 
\begin{align*}
R_{\alpha} H = (0,\ldots, 0, m^1_{t(\alpha),j_1(\alpha)},\ast, \ldots, \ast),
\end{align*}
where $1\leq j_1(\alpha)\leq K(n,d)$ and $m^1_{t(\alpha),j_1(\alpha)} = c_{t(\alpha)} z^{t(\alpha)} \neq 0$ such that $t(\alpha) \in \N^n$ with $|t(\alpha)| = K(n,d-k)$. 

Consider the minimal $\alpha_0 \in K$ w.r.t. the lexicographic order. If $m=1$, then $\alpha_0 = (2,\ldots, 2)$, and otherwise $\alpha_0 = (1,\ldots, 1)$. This implies that $j_1(\alpha)\geq j_1(\alpha_0)$ for all $\alpha>\alpha_0$.
Moreover $t(\alpha)> t(\alpha_0)$ for all $\alpha>\alpha_0$, which can be seen as follows:
Denote the monomial in $H^d_n$ at the $k$-th position by $h_{s(k),k}=a_k z^{s(k)}$, where $s(k) \in \N^n$ with $|s(k)| = K(n,d)$. 
Note that $s(j_1(\alpha)) = \alpha + t(\alpha)$ and for $\alpha>\alpha_0$ if $j_1(\alpha_0)\leq j_1(\alpha)$, then $s(j_1(\alpha_0)) \leq s(j_1(\alpha))$. Hence
\begin{align*}
\alpha_0 + t(\alpha_0) = s(j_1(\alpha_0)) \leq s(j_1(\alpha)) = \alpha + t(\alpha),
\end{align*}
which implies that $t(\alpha_0) < t(\alpha)$. Thus, in \eqref{eq:linDepHom2}, considering the coefficient of $z^{t(\alpha_0)}$ shows that $c_{\alpha_0}=0$.

Proceed inductively and assume $c_\alpha = 0$ for all $\alpha<\alpha_1, \alpha \in K$. Define $K_1 = K\setminus\{\alpha\in K: \alpha < \alpha_1\}$. Argue as above, since $\alpha_1$ is the minimal index in $K_1$ it holds that $j_1(\alpha)\geq j_1(\alpha_1)$ for all $\alpha>\alpha_1$ and the same argument as above shows that $t(\alpha)>t(\alpha_1)$. Thus $c_{\alpha_1}=0$, which proves (B) and completes the proof.
\end{proof}

\subsection{The group invariant sphere maps}

Another important class of sphere maps are the following, first introduced in \cite{DAngelo88b}:
\begin{definition}
\label{rem:groupInvSphereMap}
Define $G^\ell: \Sphere{3} \rightarrow \Sphere{2\ell+3}$ for $\ell \geq 0$ by
\begin{align*}
G^\ell(z,w) = (c_1^{\ell} z^{2\ell+1}, c_2^{\ell} z^{2\ell -1} w, \ldots, c_k^{\ell} z^{2(\ell-k) + 3}w^{k-1},\ldots  , c_{\ell+1}^{\ell} zw^\ell, c_{\ell+2}^{\ell} w^{2\ell +1}),
\end{align*}
where $c_k^{\ell}\geq 0$ for $1\leq k \leq \ell+2$ is given in \cite{DAngelo88b} or \cite{DAngeloBook}*{section 5.2.2, Theorem 9}. 
\end{definition}
The infinitesimal stabilizer of $G^\ell$ consists of the vector field $S^3_2$.
The maps $G^\ell$ are invariant under a fixed-point-free finite unitary group and appear in \cite{DKR03} as so-called sharp polynomials in the study of degree bounds for monomial maps. 


\subsection{The tensor product for infinitesimal deformations}

Similar to the case of sphere maps (\cite{DAngelo88b}, \cite{DAngelo91}*{Definition 4}) one can introduce a tensor operation for infinitesimal deformations.

Let $A \subseteq \C^n$ be a linear subspace such that $\C^n = A \oplus A^{\perp}$ is an orthogonal decomposition. For $v\in \C^n$ write $v = v_A \oplus v_{A^\perp}\in A \oplus A^{\perp}$. Similarly one can decompose the image of a map $F:\Sphere{2n-1}\rightarrow \Sphere{2m-1}$ w.r.t. $A$ and write $F = F_A \oplus F_{A^{\perp}} \in A \oplus A^{\perp}$. 

For vectors $v=(v_1,\ldots, v_n) \in \C^{n}$ and $w=(w_1,\ldots, w_m) \in \C^{m}$ the usual tensor product of $v$ and $w$ is denoted  by
\begin{align*}
v \otimes w = (v_1 w_1, \ldots, v_1 w_m, \ldots, v_n w_1, \ldots, v_n w_m) \in \C^{nm}.
\end{align*}

\begin{definition}
Let $H: \Sphere{2n-1} \rightarrow \Sphere{2m-1}$ and $G: \Sphere{2n-1} \rightarrow \Sphere{2\ell-1}$ be CR maps, $X \in \hol(H)$ and $A \subseteq \C^m$ be a linear subspace, then 
\begin{align*}
T_{(A,G)} X \coloneqq (X_A \otimes G) \oplus X_{A^{\perp}},
\end{align*}
is called the \textit{tensor product of $X$ by $G$ on $A$}.
\end{definition}

We recall that the tensor product of mappings of spheres was introduced in \cite{DAngelo88b} and \cite{DAngelo91}*{Definition 4}: For $f: \Sphere{2n-1} \rightarrow \Sphere{2m-1}$ and $g: \Sphere{2n-1} \rightarrow \Sphere{2\ell-1}$ CR maps and $A \subseteq \C^m$ a linear subspace the \textit{tensor product of $f$ by $g$ on $A$} given by $E_{(A,g)} f  = (f_A \otimes g) \oplus f_{A^{\perp}}$ is a mapping of spheres, see \cite{DAngelo91}*{Lemma 5}. An analogous result holds for infinitesimal deformations:

\begin{lemma}
\label{lem:tensorInfDef}
Let $H: \Sphere{2n-1} \rightarrow \Sphere{2m-1}$ and $G: \Sphere{2n-1} \rightarrow \Sphere{2\ell-1}$ be CR maps, $X\in \hol(H)$ and $A \subseteq \C^m$ be a linear subspace. Then $T_{(A,G)} X \in \hol(E_{(A,G)} H)$.
\end{lemma}

\begin{proof}
Set $Y = T_{(A,G)} X$ and $F = E_{(A,G)} H$. By orthogonality it holds that on $\Sphere{2n-1}$, 
\begin{align*}
\re(Y \cdot \bar F) & = \re\left( (X_A \otimes G) \cdot (\overline{H_A \otimes G}) + X_{A^\perp} \cdot \overline{H_{A^\perp}} \right) \\
& = \re\left((X_A \cdot \overline{H_A}) \| G \|^2 + X_{A^\perp} \cdot \overline{H_{A^\perp}}\right)  = \re(X \cdot \bar H) = 0, 
\end{align*}
since $\| G \|^2 = 1$ on $\Sphere{2n-1}$, hence $Y \in \hol(F)$.
\end{proof}

The next result shows that holomorphic degeneracy is preserved by tensoring.

\begin{lemma}
\label{lem:tensorHolDeg}
Let $H: \Sphere{2n-1} \rightarrow \Sphere{2m-1}$ be a CR map, $A \subseteq \C^m$ a linear subspace and $G: \Sphere{2n-1}\rightarrow \Sphere{2\ell-1}$ a CR map. If $H$ is holomorphically degenerate then $E_{(A,G)} H$ is holomorphically degenerate.
\end{lemma}

\begin{proof}
Since $H$ is holomorphically degenerate there exists a nontrivial holomorphic map $W: \C^n \rightarrow \C^m$ such that $W \cdot \bar H =0$ on $\Sphere{2n-1}$. Write $H' = E_{(A,G)}H$ and consider $W'=T_{(A,G)} W$, which is a nontrivial holomorphic vector. Then the same computation (without taking the real part) as in the proof of \cref{lem:tensorInfDef} shows that $W' \cdot \bar H'=0$ on $\Sphere{2n-1}$, i.e. $H'$ is holomorphically degenerate.
\end{proof}

\begin{example}
The converse of \cref{lem:tensorHolDeg} is not true in general:  Consider the holomorphically nondegenerate maps $H:\Sphere{3}\rightarrow \Sphere{7}, H(z,w)=(z, z w, z^2 w, w^3)$ and $G:\Sphere{3}\rightarrow \Sphere{5}, G(z,w)=(z,zw,w^2)$. Tensoring $H$ at the first component with $G$ one obtains the holomorphically degenerate map $F'(z,w) = (z^2, z^2 w, z w^2, z w, z^2 w, w^3)$. Moreover, if one applies a unitary change of coordinates and a projection $\C^6 \rightarrow \C^5$ to $F'$, one obtains $F(z,w) = (z^2, z w, \sqrt{2} z^2 w, z w^2, w^3)$, which still is holomorphically degenerate: the holomorphic vector $X=(0,-1,z /\sqrt{2}, w ,0)$ satisfies $X \cdot \bar F = 0$ on $\Sphere{3}$.
\end{example}

\begin{example}
For a sphere map $H$ its trivial infinitesimal deformations may give rise to nontrivial infinitesimal deformations of tensors of $H$: Let $H$ be the map
\begin{align*}
H^3=(z^3,\sqrt{3} z^2 w, \sqrt{3} z w^2,w^3)
\end{align*}
and $A$ be the complex subspace spanned by $(1,0\ldots, 0)\in\C^4$. Write $H=(H_1,\ldots, H_4)$ and define $X = (a,0,\ldots, 0) - (\bar a H_1) H \in \aut(H), a \in \C\setminus\{0\}$. Then it is straightforward to show that $T_{(A,H^1)} X \not \in \aut(E_{(A,H^1)} H)$.
\end{example}

\begin{example}
Following \cite{DL16}*{Prop. 2.4} one can combine two sphere maps to construct a new sphere map into a higher dimensional sphere: For $H: \Sphere{2n-1} \rightarrow \Sphere{2m-1}$ and $G: \Sphere{2n-1} \rightarrow \Sphere{2\ell-1}$ real-analytic CR maps define the \textit{juxtaposition of $H$ and $G$ with parameter $t\in [0,1]$} as the map $j_t(H,G):\Sphere{2n-1}\rightarrow \Sphere{2(m+\ell)-1}$ given by
\begin{align*}
j_t(H,G) \coloneqq \sqrt{1-t^2} H \oplus t G, \qquad t\in [0,1].
\end{align*}
Note that $j_t(H,G)$ is holomorphically degenerate: The holomorphic vector field $X =  -t H \oplus \sqrt{1-t^2} G$ satisfies $X \cdot \bar H = 0$. 
\end{example}

\section{Nondegeneracy conditions for sphere maps}
\label{sec:NonDegReflex}

In this section it is shown that holomorphic and finite degeneracy can be expressed in terms of rank conditions of the reflection matrix. 

Holomorphic nondegeneracy of a sphere map is equivalent to a generic rank condition of $V$:

\begin{proposition}
\label{prop:equivDeg}
Let $H: \Sphere{2n-1} \rightarrow \Sphere{2m-1}$ be a rational map of degree $d$.
Then the following statements are equivalent:
\begin{itemize}
\item[(a)] $H$ is holomorphically nondegenerate.
\item[(b)] There is no nontrivial holomorphic map $Y: U \rightarrow \C^m$, where $U$ is a neighborhood of $\Sphere{2n-1}$, such that $VY = 0$ on $\Sphere{2n-1}$.
\item[(c)] The matrix $V$ is generically of rank $m$ on $\Sphere{2n-1}$. 
\end{itemize}
\end{proposition}

\begin{proof}
For a map $H = \frac{P}{Q}$ the following equation holds on $\Sphere{2n-1}$:
\begin{align} 
\label{eq:V}
|Q|^2 X \cdot \bar H = Q X \cdot \bar P = Q V_HX \cdot \bar H^d_n = VX \cdot \bar H^d_n,
\end{align}
for any vector $X\in \C^m$. Let $Y: U \rightarrow \C^m$, where $U$ is a neighborhood of $\Sphere{2n-1}$, be a nontrivial holomorphic map such that $Y \cdot \bar H = 0$ on $\Sphere{2n-1}$. 
By \eqref{eq:V}, this is equivalent to $VY \cdot \bar H^d_n = 0$ on $\Sphere{2n-1}$. Since $H^d_n$ is holomorphically nondegenerate on $\Sphere{2n-1}$, the last equation is equivalent to $VY = 0$ on $\Sphere{2n-1}$. From this consideration the equivalence of (a) and (b) follows.

The equivalence of (b) and (c) holds, since (b) is equivalent to the fact that $V$ is injective on an open, dense subset of $\Sphere{2n-1}$, which is equivalent to (c). 
\end{proof}


The following proposition shows that finite nondegeneracy of a map $H$ is equivalent to a pointwise rank condition of $V_H$:

\begin{proposition}
\label{prop:equivFiniteNonDeg}
Let $H: \Sphere{2n-1} \rightarrow \Sphere{2m-1}$ be a rational map of degree $d$.
Then the following statements are equivalent:
\begin{itemize}
\item[(a)] $H$ is of degeneracy $s$ at $p \in \Sphere{2n-1}$.
\item[(b)] The kernel of the matrix $V$ is of dimension $s$ at $p \in \Sphere{2n-1}$.
\end{itemize}
In particular, the map $H$ is finitely nondegenerate at $p \in \Sphere{2n-1}$ if and only if the matrix $V$ is of rank $m$ at $p \in \Sphere{2n-1}$.
\end{proposition}

\cref{thm:rankVNonDegIntro} follows from \cref{prop:equivDeg} and \cref{prop:equivFiniteNonDeg}. 

\begin{proof}
Since $X \cdot \bar H = V_H X \cdot \bar H_n^d/\bar Q$ on $\Sphere{2n-1}$ for any $X\in \C^m$ and $V$ is holomorphic, it follows that 
\begin{align*}
X \cdot \bar L^\beta \bar H = V_H X \cdot \bar L^\beta \left(\frac{\bar H^d_n}{\bar Q}\right),
\end{align*}
on $\Sphere{2n-1}$. For any sequence of multiindices $\alpha=(\alpha_1,\ldots, \alpha_\ell), \alpha_j \in \N^{n}$ and a sphere map $F:\Sphere{2n-1} \rightarrow \Sphere{2k-1}$ define the $\ell\times k$-matrix $A^\alpha_q(\bar F) \coloneqq \left(\bar L^{\alpha_j}\bar F|_q\right)_{1\leq j \leq \ell}$ for $q\in \Sphere{2n-1}$. Then it holds that
\begin{align}
\label{eq:kernelV}
A^\beta_p(\bar H) X = A^\beta_p\left(\frac{\bar H^d_n}{\bar Q}\right)V(p)X,
\end{align}
for any $p\in\Sphere{2n-1}$, any sequence of multiindices $\beta=(\beta_1,\ldots, \beta_r), \beta_j \in \N^n$ and $r\in \N$.

Observe that $\rank A^\gamma_p(\bar H^d_n/\bar Q) = \rank A^\gamma_p(\bar H^d_n)$, for multiindices $\gamma=(\gamma_1,\ldots,\gamma_{K(n,d)}), \gamma_j\in \N^n$, chosen according to the finite nondegeneracy of $H^d_n$ given in the proof of \cref{lem:homMapNonDeg}, which can be seen as follows: One has that
\begin{align*}
\bar L^{\gamma_j}(\bar H^d_n/\bar Q) = (\bar L^{\gamma_j} \bar H^d_n)/\bar Q + \sum_{|\varepsilon| < |\gamma_j|} c_\varepsilon^{\gamma_j} \bar L^\varepsilon \bar H^d_n,
\end{align*}
where $c_\varepsilon^{\gamma_j} \in \C$ involves some constants, derivatives of $\bar Q$ and terms of the form $\bar Q^{-m_{\gamma_j,\varepsilon}}$ for some $m_{\gamma_j,\varepsilon} \in \N$. Since the first row of $A^\gamma_p(\bar H^d_n /\bar Q)$ consists of $\bar H^d_n/\bar Q$, using elementary row operations one can see that the rank of $A^\gamma_p(\bar H^d_n/\bar Q)$ agrees with the rank of $A^\gamma_p (\bar H^d_n)$.

Now it is possible to prove the equivalence of (a) and (b). Assume (a), then there is $k_0\in \N$, such that $H$ is $(k_0,s)$-degenerate at $p$. This means that $\dim_{\C} E_{k_0}'(p) = N'-s$ and hence, for any $\beta=(\beta_1,\ldots, \beta_r), \beta_j\in\N^n$ and $r\in \N$, the kernel of the matrix $A^\beta_p(\bar H)$ is at least of dimension $s$. 

Consider $\gamma=(\gamma_1,\ldots,\gamma_{K(n,d)}), \gamma_j\in \N^n$ according to the finite nondegeneracy of $H^d_n$ given in the proof of \cref{lem:homMapNonDeg}. Let $X_j$ for $1\leq j \leq s$ be linearly independent vectors in the kernel of $A^\gamma_p(\bar H)$.
By taking $\beta=\gamma$ in \eqref{eq:kernelV}, it follows that $V(p) X_j \in \ker A^\gamma_p(\bar H^d_n/\bar Q)$. Since $H_n^d$ is finitely nondegenerate in $\Sphere{2n-1}$ and by the fact that $\rank A^\gamma_p(\bar H^d_n/\bar Q) = \rank A^\gamma_p(\bar H^d_n)$, it follows that $X_j \in \ker V(p)$. Hence $\dim \ker V(p) \geq s$. 

Assume that $\dim \ker V(p) = s'>s$, i.e. there are linearly independent vectors $Y_j \in \ker V$ for $1\leq j \leq s'$. Since $H$ is of degeneracy $s$ there exists a sequence of multiindices $\delta=(\delta_1,\ldots, \delta_q), \delta_j \in \N^n$ and $q\in \N$, such that the kernel of $A^\delta_p(\bar H)$ is precisely of dimension $s$. Then $A^\delta_p(\bar H^d_n/\bar Q) V(p) Y_j = 0$ and using \eqref{eq:kernelV} with $\beta=\delta$, it follows that $A^\delta_p(\bar H)Y_j = 0$, i.e. $Y_j \in \ker A^\delta_p(\bar H)$ for $1\leq j \leq s'$, which is a contradiction to $\dim \ker A^\delta_p(\bar H) = s$.

For the other direction, assume (b) and argue similarly: If $\dim \ker V(p) = s$, consider any sequence of multiindices $\epsilon=(\epsilon_1, \ldots, \epsilon_t)$ for $\epsilon_j \in \N^n$ and $t\in \N$. Let $X_j$ for $1\leq j \leq s$ be linearly independent vectors belonging to $\ker V(p)$. By \eqref{eq:kernelV} it follows that $X_j \in \ker A^\epsilon_p(\bar H)$. Thus, the degeneracy of $H$ is at least $s$. 

Assume the degeneracy of $H$ is equal to $s'>s$. Argue as in the proof of the sufficient direction to conclude that $\dim \ker V(p) \geq s$, which is a contradiction.

The last statement follows immediately from the above shown equivalence.
\end{proof}

\begin{example} 
\label{ex:groupInvMapNonDeg}
For each $\ell \geq 0$ the map $G^\ell$ is finitely nondegenerate at $p \in \Sphere{3}$: In this case the reflection matrix $V$ is the following $(2\ell+2) \times (\ell+2)$-matrix:
\begin{align*}
V = D_1 \left(\begin{array}{ccccccc}
1 & 0 & 0 &  0 &  & &   \\
 & \tilde H^1 &  0 &  0 &  &  &    \\
 & &  \tilde H^2  & 0 &    &  &    \\
 &  &   &   \tilde H^3 &   &   &    \\
 &  &   &   &  \ddots &    &    \\
& & &    &   &  & 1
\end{array}\right) D_2,
\end{align*}
where $\prescript{t}{}{\tilde H}^k = (z^k, {\binom{k}{1}} z^{k-1} w, \ldots, w^k) \in \C^{k+1}$, blank spaces are filled up with zeros, $D_1$ is the $(2\ell+2)\times (2\ell+2)$-diagonal matrix whose nonzero entries are the reciprocals of the coefficients of $H^{2\ell +2}$ and $D_2$ is the $(\ell+2) \times (\ell+2)$-diagonal matrix which consists of the coefficients of $G^\ell$ on the diagonal. It follows that $V$ is of full rank on $\Sphere{3}$. 
\end{example}

\begin{example}
The map $H(z,w)=(z^4,z^3 w,\sqrt{3} z w, w^3)$, which sends $\Sphere{3}$ into $\Sphere{7}$, is listed in \cite{DAngelo88b}. The reflection matrix is given by
\begin{align*}
V = \left(\begin{array}{cccc}
1 & 0 & 0 & 0  \\
0 & \frac{1}{2} &  \frac{\sqrt{3}}{2} z^2 & 0  \\
0 & 0 & \sqrt{2} z w &   0 \\
0 & 0 & \frac{\sqrt{3}}{2} w^2  & \frac{1}{2} z \\
0 & 0 & 0  & w    
\end{array}\right),
\end{align*}
which is of full rank if and only if $w\neq 0$ and if $w=0$ the kernel is of dimension $1$, hence by \cref{prop:equivFiniteNonDeg} the map is finitely nondegenerate for $w\neq 0$ and of degeneracy $1$ when $w=0$. A direct computation (as in \cref{def:finiteDeg}) shows that the map is $3$-nondegenerate at points $\{z\neq0,w\neq0\}\cap \Sphere{3}$ and $4$-nondegenerate when $\{z=0,|w| = 1\}$. If $w=0$ and $|z|=1$, the map is $(3,1)$-degenerate.
\end{example}

The following example gives a map, for which the set of points in $\Sphere{3}$ where the map is $2$-degenerate consists of one isolated point.

\begin{example}
The map $H(z,w)= \bigl((az-b zw)z,(a z - b zw) w, \bar b z + \bar a z w, w^2\bigr)$, for $a,b\in\C$ satisfying $|a|^2+|b|^2=1$, sends $\Sphere{3}$ into $\Sphere{7}$. The matrix $V$ is given by
\begin{align*}
V = D\left(\begin{array}{cccc}
\bar a z & 0 & b z^2 & 0  \\
\bar a w- \bar b & \bar a z &  a z + 2 b zw & 0   \\
0 & \bar a w -b & a w + b w^2 &   z \\
0 & 0 & 0  & w  
\end{array}\right),
\end{align*}
where $D$ is the $4\times 4$-diagonal matrix whose nonzero entries are the reciprocals of the  coefficients of $H^3$. Assuming $a=b=\frac{1}{\sqrt{2}}$, it holds that $V$ is of full rank if $Y_1 \coloneqq \{z\neq 0,w\neq 0\} \cap \Sphere{3}$ and the complex dimension of the kernel is $1$ if $Y_2\coloneqq \bigl\{\{z=0,w\neq 1\} \cup \{z\neq 0,w=0\} \bigr\}\cap \Sphere{3}$. The kernel of $V$ is of complex dimension $2$ at $p_0\coloneqq (0,1)\in \Sphere{3}$. It can be shown by a direct computation (as in \cref{def:finiteDeg}) that $H$ is $3$-nondegenerate in $Y_1$, $(2,1)$-degenerate in $Y_2$ and $(1,2)$-degenerate at $p_0$.
\end{example}

The following result gives conditions to guarantee that a sphere map is finitely degenerate:

\begin{corollary}
\label{cor:degreeKernel}
If a rational sphere map $H: \Sphere{2n-1} \rightarrow \Sphere{2m-1}$ of degree $d$ satisfies $K(n,d) < m$, then $H$ is finitely degenerate at any $p \in \Sphere{2n-1}$. In particular the map is holomorphically degenerate.
\end{corollary}

\begin{proof}
The $K(n,d)\times m$-matrix $V$ satisfies $\rank V \leq \min(K(n,d),m)=K(n,d)$ on $\Sphere{2n-1}$.
If $H$ would be finitely nondegenerate at $p \in \Sphere{2n-1}$, by \cref{prop:equivFiniteNonDeg}, $V$ would be injective at $p$, hence $\rank V = m$ at $p$, 
a contradiction. By \cref{prop:equivDeg} it follows that $H$ is holomorphically degenerate. 
\end{proof}

\begin{example}
The map $H(z,w) = (z, \cos(t)w, \sin(t)zw, \sin(t) w^2), t\in [0,2\pi)$, sends $\Sphere{3}$ to $\Sphere{7}$ and is holomorphically degenerate by \cref{cor:degreeKernel}. The reflection matrix $V$ is given as follows:
\begin{align*}
V= \left(
\begin{array}{cccc}
z & 0 & 0 & 0 \\
\frac{w}{\sqrt{2}} & \frac{\cos(t) z}{\sqrt{2}} & \frac{\sin(t)}{\sqrt{2}} & 0 \\
0 & \cos(t) w & 0 & \sin(t)
\end{array}
\right).
\end{align*}
If $\sin(t) \neq 0$, then $X(z,w)=(0,1,- \cot(t) z, -\cot(t) w)$ is a holomorphic vector field tangent to $\Sphere{7}$ along the image of $H$.
One can check that if $\cos(t),\sin(t)\neq 0$ the map is of degeneracy $1$ for $z\neq 0$ and of degeneracy $2$ if $z=0$. If $\cos(t) \neq 0$ and $\sin(t) = 0$ the map is of degeneracy $2$ and when $\cos(t)=0$ the map is $1$-degenerate.

\end{example}

The set of points where the map is finitely degenerate can be described by using \cref{prop:equivFiniteNonDeg}:

\begin{corollary}
\label{cor:complexVariety}
Let $H: \Sphere{2k-1} \rightarrow \Sphere{2m-1}$ be of generic degeneracy $s$ in $\Sphere{2k-1}$. The set of points in $\Sphere{2k-1}$, where $H$ is of degeneracy $s'>s$ is contained in a complex algebraic variety intersecting $\Sphere{2k-1}$.
\end{corollary}

\begin{proof}
The set $D$ of points where $H$ is of degeneracy $s'> s$ is the complement $Y$ of the set where $H$ is of generic degeneracy $s$, which is given by the union of the zero sets of any minor of $V$ of size strictly less than $\rank V$. 
Since $V$ consists of holomorphic polynomial entries, $Y$ is a complex algebraic variety and, by \cref{prop:equivFiniteNonDeg}, agrees with $D$. 
\end{proof}

Note that \cref{prop:equivFiniteNonDeg} shows that \cref{cor:degreeKernel} and \cref{cor:complexVariety} are equivalent to \cite{DAngelo03}*{Corollary 4.4} and \cite{DAngelo03}*{Corollary 4.2} respectively.

\section{The X-variety of a sphere map}
\label{sec:XVariety}

In this section sufficient and necessary conditions in terms of nondegeneracy conditions are provided to guarantee that the X-variety of a sphere map satisfies certain properties, such as agreeing with the graph of the map or being an affine bundle.

First, the general definition of the X-variety of a map is repeated for the reader's convenience, see \cite{Forstneric89} and \cite{DAngelo03}:

\begin{definition}
\label{def:XVariety}
Let $M\subset \C^N$ and $M'\subset \C^{N'}$ be real-analytic hypersurfaces and $H: M \rightarrow M'$ be a real-analytic CR map. Let $p\in M$ and $p' = H(p)$. Assume $M \cap U = \{Z \in U: \rho(Z,\bar Z) = 0\}$ and $M'\cap U' = \{Z' \in U': \rho'(Z',\bar Z') = 0\}$, where $U \subset\C^N$ and $U'\subset \C^{N'}$ are neighborhoods of $p$ and $p'$ and $\rho$ and $\rho'$ are real-analytic defining functions for $M$ and $M'$ defined in $U$ and $U'$ respectively. Define the following set
\begin{align*}
X_H \coloneqq \{(W,W') \in \C^{N+N'}: \rho'(H(Z),\bar W')=0 \text{ if } \rho(Z,\bar W) = 0\},
\end{align*}
which is called the \textit{X-variety of $H$} near $p \in M$.
\end{definition}

Since $H$ maps $M$ into $M'$ it follows that $(Z,H(Z)) \in X_H$, i.e. the graph of $H$ is contained in $X_H$. In \cite{DAngelo03}*{Theorem 4.1} it is shown in the case when $M \subset \C^n$ and $M'\subset \C^m$ are unit spheres that for any $z\neq 0$ it holds that $(z,z') \in X_H$ if and only if $z'-H(z) \in \ker V(z)$. 
$X_H$ has an \textit{exceptional fiber at $p\in \Sphere{2n-1}$} if the dimension of the fiber $\{p' \in \C^{N'}: (p,p')\in X_H\}$ exceeds its generic value. In \cite{DAngelo03}*{Corollary 4.2} it is argued that the set of points over which $X_H$ has an exceptional fiber agrees with the set of points $p\in \Sphere{2n-1}$ where the rank of $V(p)$ drops.

Moreover, in \cite{DAngelo03}*{Theorem 4.1} the following properties of $X_H$ are proved:
\begin{itemize}
\item[(a)] $X_H$ is an affine bundle over $\C^n \setminus\{0\}$ if and only if the rank of $V(z)$ is constant for each $z\neq 0$ in the domain of $H$.
\item[(b)] $X_H$ equals the graph of $H$ if and only if, for each $z\neq 0$ in the domain of $H$, the null space of $V(z)$ is trivial.
\end{itemize}

Using the facts from \cref{sec:NonDegReflex}, relating nondegeneracy conditions and rank conditions of the reflection matrix, the following characterizations hold:

\begin{theorem}
\label{thm:XVariety}
Let $H: \Sphere{2n-1} \rightarrow \Sphere{2m-1}$ be a real-analytic CR map. Then the following statements hold:
\begin{itemize}
\item[(a)] $X_H$ is an affine bundle over $\C^n\setminus\{0\}$ if and only if $H$ is of finite degeneracy $s$ at any point of $\Sphere{2n-1}$.
\item[(b)] $X_H$ equals the graph of $H$ if and only if $H$ is finitely nondegenerate at any point of $\Sphere{2n-1}$.
\item[(c)] $X_H$ has an exceptional fiber at $p\in \Sphere{2n-1}$ if and only if $H$ is not of generic degeneracy $s(H)$ at $p\in \Sphere{2n-1}$.
\end{itemize}
\end{theorem}

\begin{proof}
The proofs of (a) and (b) follow from \cref{prop:equivFiniteNonDeg} and the characterizations from \cite{DAngelo03}*{Theorem 4.1} stated above. Since the rank conditions involved are constant on $\Sphere{2n-1}$, they also hold in a neighborhood of $\Sphere{2n-1}$, to which $H$ extends.
For (c) note that the points where $H$ is of generic degeneracy $s(H)$ (see the remark after \cref{def:finiteDeg}) form an open dense subset $S$ of $\Sphere{2n-1}$. Hence in the complement of $S$ the degeneracy of $H$ is strictly bigger and by \cref{prop:equivFiniteNonDeg} the rank of the reflection matrix is strictly smaller. Thus, the complement of $S$ is precisely the set where $X_H$ possesses an exceptional fiber.
\end{proof}

\section{Infinitesimal deformations of sphere maps}

In this section infinitesimal deformations of rational sphere maps are studied. It turns out that similarly as in the case of sphere maps, where each sphere map is related to the homogeneous sphere map by tensoring, infinitesimal deformations of a sphere map are related to infinitesimal deformations of the homogeneous sphere map by the reflection matrix. 

\begin{lemma}
\label{lem:infDefRational}
Let $H = \frac P Q: U \rightarrow \Sphere{2m-1}$ be a holomorphically nondegenerate rational sphere map of degree $d$, where $U$ is a neighborhood of $\Sphere{2k-1}$. Then each $X \in \hol(H)$ is of the form $X = \frac{X'}{Q}$, where $X'$ is a holomorphic polynomial of degree at most $2d$ satisfying $\re(X' \cdot \bar P) = 0$ on $\Sphere{2k-1}$.
\end{lemma}

\begin{proof}
Let $H$ be given as in the assumption of the lemma, where $P=(P_1, \ldots, P_m)$ and $Q: U \rightarrow \C$ with $Q \neq 0$ on $U$, a neighborhood of $\Sphere{2k-1}$. Then $|Q|^2\re(X \cdot \bar H) = \re(Q X \cdot \bar P)$. Set $X' \coloneqq Q X$. Considering homogeneous expansions of $X' = \sum_{\ell\geq 0} {X'}^\ell$ and $P = \sum_{j=0}^d P^j$ one obtains the following equation:
\begin{align*}
\sum_{\ell \geq 0} \sum_{j=0}^d {X'}^\ell \cdot {\bar P}^j + \sum_{r\geq 0} \sum_{s=0}^d \bar  X'^r \cdot P^s = 0.
\end{align*}
After setting $Z\mapsto Z e^{i t}$ for $t\in \R$ collect the Fourier coefficient of degree $d+\ell_0$ for $\ell_0 \geq 1$ to get: 
\begin{align*}
\sum_{j=0}^d {X'}^{d+\ell_0+j} \cdot {\bar P}^j = 0.
\end{align*}
By the holomorphic nondegeneracy of $H$ this implies that $X'^{\ell} \equiv 0$ for $\ell \geq 2 d+1$, i.e. $\deg X' \leq 2 d$.
\end{proof}

Denote by $\mathcal P^d(k,m)$ the space of complex polynomial maps from $\C^k$ to $\C^m$ of degree $d$ with $\dim_{\R} \mathcal P^d(k,m) = 2 m \sum_{\ell=0}^d\binom{\ell + k -1}{\ell}$. The following definition is justified by the previous \cref{lem:infDefRational} and in fact $\hol(H) $ can be identified with a space of polynomial maps.

\begin{definition}
Let $H = \frac P Q: U \rightarrow \Sphere{2m-1}$ be a holomorphically nondegenerate rational sphere map of degree $d$, where $U$ is a neighborhood of $\Sphere{2k-1}$. Define $\dim \hol(H) \coloneqq \dim_{\R} \{X' \in \mathcal P^{2d}(k,m): \frac{X'}{Q} \in \hol(H)\}$.
\end{definition}

In \cite{DAngelo91} and \cite{DAngeloBook}*{section 5.1.4, Theorem 4} it is shown that for any polynomial sphere map of degree $d$ if one applies finitely many tensoring operations to it one obtains the homogeneous sphere map of degree $d$. Moreover in \cite{DAngeloBook}*{section 5.1.4, Theorem 3} it is shown that the homogeneous sphere map is up to a unitary transformation unique among all polynomial and homogeneous sphere maps. The following theorem gives the corresponding results in terms of infinitesimal deformations.

\begin{theorem}
\label{thm:polyMapInfDef}
Let $H: \Sphere{2n-1} \rightarrow \Sphere{2m-1}$ be a holomorphically nondegenerate rational map of degree $d$, then $\dim \hol(H) \leq \dim \hol(H^d_n)$. \\
If $H: \Sphere{2n-1} \rightarrow \Sphere{2m-1}$ is a polynomial map of degree $d$, it holds that $\dim \hol(H) = \dim \hol(H^d_n)$ if and only if $H$ is unitarily equivalent to $H^d_n$.
\end{theorem}


\begin{proof}
Let $H = \frac P Q: \Sphere{2n-1} \rightarrow \Sphere{2m-1}$ be a rational map with $Q\neq 0$ on $\Sphere{2n-1}$. 
Consider as in \cref{def:VH} the matrix $V: \C^m \rightarrow \C^{K(n,d)}$ whose entries are holomorphic polynomials in $z\in \C^n$. Then it holds on $\Sphere{2n-1}$ that $|Q|^2 X \cdot \bar H =  VX \cdot \bar H^d_n$ for $X\in \C^m$ as in \eqref{eq:V}. Thus on $\Sphere{2n-1}$ one obtains,
\begin{align}
\label{eq:holHFromHolHd}
X \in \hol(H) \Leftrightarrow VX \in \hol(H^d_n).
\end{align}
By \cref{lem:infDefRational} there are polynomials $X_1', \ldots, X_k' \in \mathcal P^{2d}(n,m)$ such that $\{X_j = \frac{X'_j}{Q}:1 \leq j \leq k\}$ is a basis of $\hol(H)$. 
From \cref{prop:equivDeg} it follows that $\{VX_j: 1\leq j \leq k\}$ is a set of linearly independent polynomials in $\hol(H^d_n)$ which implies $k \leq \dim\hol(H^d_n)$. 

To show the nontrivial implication of the second claim, assume that $H=P$ is polynomial of degree $d$ and $\dim \hol(P) = \dim \hol(H^d_n)$. 
By \cref{prop:equivDeg}, since the reflection matrix $V$ is injective on a dense, open subset $S$ of $\Sphere{2n-1}$ as a map from $\hol(P)$ 
to $\hol(H^d_n)$, 
it follows by the rank theorem that $\dim V(\hol(P))= \dim \hol(P)$. 
Using the assumption $\dim \hol(P) = \dim \hol(H^d_n)$ this implies that $V$ is invertible as a map from $\hol(P)$ to $\hol(H^d_n)$, for $z\in S$. 
Thus, for any $Y\in \hol(H^d_n)$ there exists $X\in \hol(P)$ with $VX = Y$, such that on $S$, the following equation holds:
\begin{align*}
Y \cdot \bar H^d_n = VX \cdot \bar H^d_n = X \cdot \bar P = V^{-1}Y \cdot \bar P  = Y \cdot \prescript{t}{}{V^{-1}} \bar P.
\end{align*}
Choosing $Y = i H^d_n \in \aut(H^d_n) \subset \hol(H^d_n)$ in the previous equation, it becomes after using $H^d_n \cdot \bar H^d_n = 1$ on $\Sphere{2n-1}$:
\begin{align}
\label{eq:VHP}
1 = H^d_n \cdot \prescript{t}{}{V^{-1}} \bar P = V^{-1}H^d_n \cdot \bar P.
\end{align}
Note that the matrix $B \coloneqq V^{-1}$ depends holomorphically on $z \in U$, where $U$ is an open set in $\C^n$, such that $U \cap \Sphere{2n-1} = S$. 
Consider the homogeneous expansion of $B = \sum_{k\geq 0}B_k$ and $P=\sum_{\ell=0}^d P_\ell$. Take $z \mapsto e^{i t} z$ in \eqref{eq:VHP}, and collect Fourier coefficients. Looking at  the constant Fourier coefficient we see that 
\begin{align*}
1 = B_0 H^d_n \cdot \bar P_d,
\end{align*}
on $S$. This equation can be rewritten as in the proof of \cite{DAngeloBook}*{section 5.1.4, Theorem 3} as follows,
\begin{align*}
0 = H^d_n\cdot \bar H^d_n - B_0 H^d_n \cdot \bar P_d = H^d_n \cdot (\bar H^d_n - \prescript{t}{}{B_0} \bar P_d),
\end{align*}
such that the holomorphic nondegeneracy of $H^d_n$ implies that $H^d_n = \prescript{t}{}{\bar B_0} P_d$.
By some linear algebra 
this shows that $P_d = U H^d_n$, where $U$ is a unitary matrix. 
Write $P = U H^d_n + F$, where $F$ is a holomorphic polynomial of degree $d-1$.
Since $P$ maps $\Sphere{2n-1}$ to $\Sphere{2K(n,d)-1}$ it holds on $\Sphere{2n-1}$ that,
\begin{align*}
1 = \|P\|^2 = H^d_n \cdot \bar H^d_n + U H^d_n \cdot \bar F + \bar U \bar H^d_n \cdot F + F \cdot \bar F,
\end{align*}
hence after setting $F' = \prescript{t}{}{\bar U} F$, one obtains
\begin{align}
\label{eq:mapEqHomPlus}
0 = H^d_n \cdot \bar F' + \bar H^d_n \cdot F' + F' \cdot \bar F'.
\end{align}
Consider $z\mapsto e^{i t} z$ and a homogeneous expansion of $F'=\sum_{j=0}^{d-1} F_j'$ and collect the coefficient of $e^{i d t}$ to get that $H^d_n \cdot \bar F'_0 = 0$, hence by the holomorphic nondegeneracy of $H^d_n$ one obtains $F_0' = 0$. Proceed inductively to show that $F_k' = 0$ for $k \leq d-1$. Assume that $F'_\ell = 0$ for all $0\leq \ell \leq k-1$. Collect the coefficient of $e^{i(d-k)t}$ in \eqref{eq:mapEqHomPlus} to obtain that,
\begin{align*}
0 = H^d_n \cdot \bar F'_k + \sum_{m=0}^{k-1} F'_{d-k+m} \cdot \bar F'_m,
\end{align*}
which, by using the induction hypothesis and the holomorphic nondegeneracy of $H^d_n$, implies that $F'_{k} = 0$. In total one obtains that $P$ is unitarily equivalent to $H^d_n$ in $S$, hence they are equivalent everywhere on $\Sphere{2n-1}$. 
\end{proof}

One has the following inequality for the dimension of the space of infinitesimal deformations, when the tensor product is involved.

\begin{corollary}
\label{cor:mono}
Let $A \subseteq \C^m$ be a complex subspace, $H: \Sphere{2n-1} \rightarrow \Sphere{2m-1}$ and $G: \Sphere{2n-1} \rightarrow \Sphere{2\ell-1}$ be non-constant real-analytic CR maps. Assume $F = E_{(A,G)}H$. Then $\dim \hol(H) \leq \dim \hol(F)$ and if $H$ is holomorphically nondegenerate and equality holds if and only if $F = H$. 
\end{corollary}

\begin{proof}
If $F$ is holomorphically degenerate, by \cref{prop:propNondeg}, the inequality is satisfied. Assume that $F$ is holomorphically nondegenerate, then the same holds for $H$ by \cref{lem:tensorHolDeg}.
Instead of using $V$ as in the proof of \cref{thm:polyMapInfDef}, one considers $\widetilde V$ a linear map defined by $\widetilde V(X) \coloneqq T_{(A,G)} X$. Then $X \cdot \bar H = \widetilde V(X) \cdot \bar F$, i.e. $X \in \hol(H) \Leftrightarrow \widetilde V(X) \in \hol(F)$. 

It holds that if there exists $Y \in \hol(H)$ with $\widetilde V(Y) = 0$ on $\Sphere{2n-1}$, then $0 = \widetilde V(Y) \cdot \bar F = Y \cdot \bar H$ on $\Sphere{2n-1}$, which implies, since $H$ is holomorphically nondegenerate, that $Y\equiv 0$.
From this it follows that the set $\{\widetilde V(X_j): 1\leq j \leq k\}$, for $X_1,\ldots, X_k$ a basis of $\hol(H)$, is linearly independent in $\hol(F)$, which gives the claimed inequality.

For the equality, assume that $\dim  \hol(F) = \dim \hol(H) < \infty$. Then $\dim \hol(H) = \dim \widetilde V(\hol(H)) \leq \dim \hol(F) = \dim \hol(H)$, which implies, as in the proof of \cref{thm:polyMapInfDef}, that on $\Sphere{2n-1}$ the map $\widetilde V$, as a map from $\hol(H)$ to $\hol(F)$, is invertible. 
Using a similar argument as in the proof of \cref{thm:polyMapInfDef} (replacing $V$ by $\tilde V$, $H^d_n$ by $F$ and $P$ by $H$ and using $X \cdot \bar H = \tilde V(X) \cdot \bar F$) it follows that $H$ and $F$ are unitarily equivalent, which can only happen, when the complex subspace $A$ is trivial. This concludes the proof.
\end{proof}

The remainder of this section is a collection of lemmas concerning some properties of $V_H$ and its transpose and provide sufficient and necessary conditions for infinitesimal rigidity in terms of $V_H$ and its adjoint.

\begin{proposition}
\label{prop:TensorVHMap}
For any polynomial map $H:\Sphere{2n-1} \rightarrow \Sphere{2m-1}$ of degree $d$ one has $V_H H = H^d_n$ and $H = \prescript{t}{}{\bar V_H} H^d_n$ on $\Sphere{2n-1}$.
\end{proposition}

\begin{proof}
Using the matrix $V_H$ from \cref{def:VH} the following holds on $\Sphere{2n-1}$:
\begin{align*}
V_HX \cdot \bar H^d_n = X \cdot \bar H, 
\end{align*}
for all $X\in \C^m$. Taking $X = H$ in the above equation and using $H \cdot \bar H = 1 = H^d_n \cdot \bar H^d_n$ on $\Sphere{2n-1}$, it holds that,
\begin{align*}
V_H H \cdot \bar H^d_n = 1 = H^d_n \cdot \bar H^d_n, 
\end{align*}
on $\Sphere{2n-1}$. Note that $V_H H$ has holomorphic components such that the holomorphic nondegeneracy of $H^d_n$ implies that $V_H H = H^d_n$.

For the other identity one has,
\begin{align*}
X \cdot \bar H = V_H X \cdot \bar H^d =  X \cdot \prescript{t}{}{V_H} \bar H^d,
\end{align*}
for all $X \in \C^n$, which concludes the proof.
\end{proof}

The following example shows that a similar relation as the second identity in \cref{prop:TensorVHMap} does not hold for infinitesimal deformations in general:
\begin{example}
\label{ex:tBVGlNotHolo}
Let $X$ be of the form as $T_1$ in \cref{sec:generalInfAuto} with $H$ being the map
\begin{align*}
G^3=(z^7,\sqrt{7} z^5 w, \sqrt{14} z^3 w^2, \sqrt{7} z w^3, w^7)
\end{align*}
and $\alpha' = (0,0,0,a,0) \in\C^5$. Then, on $\Sphere{3}$, one has:
\begin{align*}
\prescript{t}{}{\bar V_{G^3}} V_{G^3} X = & \frac a 5 \left(0,0,\sqrt{2} z^2 \bar w(2 |z|^2 + 3 |w|^2), 5 |w|^6 + 15 |z|^2 |w|^4 + 9 |z|^4 |w|^2 + |z|^6,0\right) \\
& + \text{holomorphic terms},
\end{align*}
which does not extend holomorphically to a neighborhood of $\Sphere{3}$.
\end{example}

\begin{lemma}
\label{lem:UntensorVH}
Let $H:\Sphere{2n-1} \rightarrow \Sphere{2m-1}$ be a polynomial map of degree $d$.
\begin{itemize}
\item[(a)] It holds that $X\in \hol(H)$ if and only if $X'\coloneqq \prescript{t}{}{\bar V_H} V_H X$ satisfies $\re(X' \cdot \bar H) = 0$ on $\Sphere{2n-1}$. 
\item[(b)] If $Y \in \hol(H^d_n)$ has the property that $\prescript{t}{}{\bar V_H} Y$ is holomorphic, then $Y \in V_H(\hol(H))$. 
\end{itemize}
\end{lemma}

In (b) the necessary direction need not be true as \cref{ex:tBVGlNotHolo} shows.

\begin{proof}
By \cref{prop:TensorVHMap} the following holds on $\Sphere{2n-1}$:
\begin{align*}
X \cdot \bar H = V_H X \cdot \bar H^d_n = V_H X \cdot \bar V_H \bar H = \prescript{t}{}{\bar V_H} V_H X \cdot \bar H,
\end{align*}
which shows (a). 
In (b) assume $X\coloneqq \prescript{t}{}{\bar V_H} Y$ is holomorphic. By \cref{prop:TensorVHMap} one has,
\begin{align*}
Y \cdot \bar H^d_n = Y \cdot \bar V_H \bar H = \prescript{t}{}{\bar V_H} Y \cdot \bar H = X \cdot \bar H, 
\end{align*}
on $\Sphere{2n-1}$ and taking the real part shows that $X \in \hol(H)$. Consider the above equation and note that one has $ X \cdot \bar H = V_H X \cdot \bar H^d_n$, such that, since $V_H X$ is holomorphic, the holomorphic nondegeneracy of $H^d_n$ implies $Y=V_H X \in V_H(\hol(H))$.
\end{proof}

\begin{proposition}
Let $H:\Sphere{2n-1} \rightarrow \Sphere{2m-1}$ be a polynomial map of degree $d$. 
\begin{itemize}
\item[(a)] If the map $H$ is infinitesimally rigid then $\prescript{t}{}{\bar V_H} \hol (H^d_n)\cap \hol(H)= \prescript{t}{}{\bar V_H} V_H \aut(H)\cap \hol(H)$. 
\item[(b)] Assume that the map $H$ is holomorphically nondegenerate. If $\prescript{t}{}{\bar V_H} \hol (H^d_n)= \prescript{t}{}{\bar V_H} V_H \aut(H)$, then $H$ is infinitesimally rigid.
\end{itemize}
\end{proposition}

\begin{proof}
To prove (a), assume $X \in \prescript{t}{}{\bar V_H} \hol (H^d_n)\cap \hol(H)$ such that there exists $Y \in \hol(H^d_n)$ with $X = \prescript{t}{}{\bar V_H} Y$. Thus, $\prescript{t}{}{\bar V_H} Y\in \hol(H)$ (in particular $\prescript{t}{}{\bar V_H} Y$ is holomorphic) and by \cref{lem:UntensorVH} (b) it holds that  $Y \in V_H \hol (H) = V_H \aut(H)$ by the infinitesimal rigidity of $H$. In total this shows that $X \in \prescript{t}{}{\bar V_H} V_H \aut(H) \cap \hol(H)$.
For the other implication, note that if $X \in \prescript{t}{}{\bar V_H} V_H \aut(H) \cap \hol(H)$, then, since $V_H\aut(H) \subset \hol(H^d_n)$, it follows that $X \in \prescript{t}{}{\bar V_H} \hol(H^d_n) \cap \hol(H)$.

For (b) let $X \in \hol(H)$, then $V_H X \in \hol(H^d_n)$ and hence $\prescript{t}{}{\bar V_H} V_H X \in  \prescript{t}{}{\bar V_H} \hol(H^d_n) = \prescript{t}{}{\bar V_H} V_H \aut(H)$. 
Thus there exists $T \in \aut(H)$, such that $A X = A T$ for $A\coloneqq \prescript{t}{}{\bar V_H} V_H$. Since $H$ is holomorphically nondegenerate it holds that the $K(n,d) \times m$-matrix $V_H$ is injective on a dense open subset $S$ of $\Sphere{2n-1}$ (see \cref{prop:equivDeg}), hence one has $\rank V_H = m$  in $S$. Since $V_H$ consists of holomorphic entries in $z$, this means that $V_H$ is of rank $m$ in an open set $U$ such that $U\cap\Sphere{2n-1}=S$. It follows that the $m\times m$-matrix $A$ is of full rank $m$ in $U$, and thus $X = T$ in $U$. Since $X$ and $T$ are holomorphic they agree in $\C^n$. This shows that $X \in \aut(H)$.
\end{proof}

\section{Infinitesimal deformations of the homogeneous sphere map}

In this section the dimension of the space of infinitesimal deformations of the homogeneous sphere map $H^d_n$ (see \cref{def:homSphereMap}) is computed. 

\begin{theorem}
\label{thm:dimInfDef}
The real dimension of the space of infinitesimal deformations of $H^d_n$ is given by $\left(\frac{2d+n}{d}\right)K(n,d)^2$. 
\end{theorem}

\begin{proof}
Define the map $\hat H_n^d(z)=(z^\alpha)_{|\alpha|=d} \in \C^{K(n,d)}$ and write $Y = D_n^d X \in \hol(H_n^d)$, where $D_n^d$ is given as in the hypothesis. Then $X$ has to satisfy the following equation on $\Sphere{2n-1}$,
\begin{align*}
X \cdot \overline{\hat H_n^d} + \bar X \cdot \hat H_n^d  = 0.
\end{align*}
In the above equation consider the homogeneous expansion of $X = \sum_{k\geq 0} X^k$, where $X^k\in \C^{K(n,d)}$ is a homogeneous polynomial in $z$ of order $k \geq 0$. Change coordinates via $z \mapsto e^{i \theta}z$ for $\theta \in \R$ as in \cite{DAngelo88b}*{Lemma 16} to obtain after shifting indices, on $\Sphere{2n-1}$:
\begin{align}
\label{eq:fourierCoeff}
\sum_{s \geq -d} X^{d+s} \cdot \overline{\hat H^d_n} e^{i s \theta}+ \sum_{t\leq d}\bar X^{d-t} \cdot \hat H^d_n e^{i t \theta} = 0.
\end{align}
This implies $X^{\ell} \equiv 0$ for $\ell \geq 2 d + 1$, such that $-d \leq s,t\leq d$. 
Collect Fourier coefficients of $e^{i r \theta}$ for $-d \leq r \leq d$ in \eqref{eq:fourierCoeff} such that for $X \in \hol(H^d_n)$ it is enough to study the solutions of the following equations:
\begin{align}
\label{eq:infDefNonHom}
X^{d+r} \cdot \overline{\hat H^d_n} + \bar X^{d-r} \cdot \hat H^d_n = 0, \qquad 0 \leq r \leq d, 
\end{align}
on $\Sphere{2n-1}$. For $r=0$ the equation is real and will be treated below. Fix $1\leq r \leq d$ and homogenize \eqref{eq:infDefNonHom} as in \cite{DAngelo91}*{section II} by multiplying the second term with $\|z\|^{2r}$, such that the following equation holds for all $z\in \C^n$:
\begin{align}
\label{eq:infDefHom}
\overline{\hat H^d_n} \cdot X^{d+r}  + \sum_{|\alpha| = r} {\binom{r}{\alpha}} \bar X^{d-r} \bar z^\alpha \cdot  z^\alpha \hat H^d_n = 0.
\end{align}
Write $X^{d-r} = B^{d,r}_n \hat H^{d-r}_n$, where $B_n^{d,r}$ is a $K(n,d)\times K(n,d-r)$-matrix and write $z^\alpha \hat H^d_n =  C_n^{d,\alpha} \hat H^{d+r}_n$, where $C_n^{d,\alpha}$ is a $K(n,d) \times K(n,d+r)$-matrix whose entries consist of $0$'s or $1$'s. Rewrite the second term of \eqref{eq:infDefHom} as
\begin{align*}
\sum_{|\alpha| = r} {\binom{r}{\alpha}} \bar B^{d,r}_n \overline{\hat H^{d-r}_n} \cdot  C_n^{d,\alpha} H^{d+r}_n = \sum_{|\alpha| = r} {\binom{r}{\alpha}} \bar z^\alpha \overline{\hat H^{d-r}_n} \cdot  \prescript{t}{}{\bar B^{d,r}_n} C_n^{d,\alpha} H^{d+r}_n 
\end{align*}
Write $z^\alpha \hat H^{d-r}_n =  D_n^{d,\alpha} \hat H^{d-r}_n$, where $D_n^{d,\alpha}$ is a $K(n,d) \times K(n,d-r)$-matrix whose entries consist of $0$'s or $1$'s, such that using the holomorphic nondegeneracy of $H^d_n$, one obtains 
\begin{align}
\label{eq:infDefHomHol}
X^{d+r}  +  \sum_{|\alpha| = r} {\binom{r}{\alpha}} D_n^{d,\alpha} \prescript{t}{}{\bar B^{d,r}_n} C_n^{d,\alpha} H^{d+r}_n  = 0.
\end{align}
Setting $X^{d+r} = A_n^{d,r} \hat H_n^{d+r}$, where $A_n^{d,r}$ is an $K(n,d) \times K(n,d+r)$-matrix, \eqref{eq:infDefHomHol} gives:
\begin{align*}
A_n^{d,r}  = - \sum_{|\alpha| = r} {\binom{r}{\alpha}} 
D_n^{d,\alpha} \prescript{t}{}{\bar B^{d,r}_n} C_n^{d,\alpha}.
\end{align*}

This implies that the solutions of \eqref{eq:infDefNonHom} depend on the entries of $B_n^{d,r}$ for $1\leq r \leq d$. The number of real entries in $B_n^{d,r}$ for $1\leq r \leq d$ is given by
\begin{align}
\label{eq:dimOne}
2 \sum_{r = 1}^{d} K(n,d)K(n,d-r) = 2 K(n,d) \binom{n+d-1}{n}.
\end{align}

For $r=0$ in \eqref{eq:infDefNonHom} one has 
\begin{align}
\label{eq:infDefHomR0}
X^{d} \cdot \overline{\hat H^d_n} + \bar X^{d} \cdot \hat H^d_n = 0,
\end{align}
which is a homogeneous equation and hence holds on $\C^n$. Writing $X^d=B_n^{d,0} \hat H^d_n$, where $B_n^{d,0}$ is a $K(n,d) \times K(n,d)$-matrix, in \eqref{eq:infDefHomR0} shows that $\prescript{t}{}{B_n^{d,0}} + \bar B_n^{d,0} = 0$. Hence $B_n^{d,0}$ depends on $K(n,d)^2$ real parameters. 
Moreover the stabilizer of $H^d_n$ consists of $S_n^2$ and $S_n^3$ and the remaining elements of $\aut(H_n^d)$ coming from $\hol(\Sphere{2n-1})$ appear in \eqref{eq:infDefNonHom} for $r=1$ and have been taken into account in \eqref{eq:dimOne}. 
In total, adding \eqref{eq:dimOne} and the number of real solutions of \eqref{eq:infDefNonHom} for $r=0$ gives the claimed dimension of $\hol(H^d_n)$, which finishes the proof.
\end{proof}

\cref{thm:polyMapInfDefIntro} is an immediate consequence of \cref{thm:polyMapInfDef} and \cref{thm:dimInfDef}. Some examples illustrating \cref{thm:dimInfDef} are provided in the following:

\begin{example}
\label{ex:listInfDefHomMap}
For $n=2$ write $H_2^d(z,w)=(a_1^d z^d, a_2^d z^{d-1}w, \ldots, a_{d+1}^d w^d)$, where $(a_k^d)^2 = \binom{d}{k-1}$ for $1\leq k \leq d+1$, such that $\dim_{\R} \hol(H^d_2) = (d+1)^3$ and any $X\in \hol(H^d_2)$ is a real linear combination of infinitesimal deformations $N^{m,k}_{d,r}$ of the following form: 
\begin{align*}
N^{m,k}_{d,r} = D_{d+1} \left(C^{m,k}_{d,r} \hat H^{d-r} - \sum_{\ell=0}^r {\binom{r}{\ell}} \left(0_\ell, z^{r-\ell} w^\ell \prescript{t}{}{\bar C}^{m,k}_{d,r} \hat H^d,0_{r-\ell} \right) \right),
\end{align*}
where $0\leq r\leq d$, $D_{d+1}$ is the $(d+1)\times(d+1)$-diagonal matrix with entries $1/a_k^d$ for $1\leq k \leq d+1$, $C^{m,k}_{d,r} =(c_{ij})$ is a $(d+1) \times (d+1-r)$-matrix with $c_{mk} \in \{1,i\}$ for $1\leq m \leq d+1$ and $1\leq k \leq d+1-r$ and all other entries are $0$ and $0_j$ denotes the zero-vector in $\C^j$. 
To illustrate which nontrivial infinitesimal deformations for $H^d_2$ appear, a list for $d=2,3$ is given. Note that if $N$ is a nontrivial infinitesimal deformation of $H^d_2$, then $\tilde N = \phi' \circ N \circ \phi$ is again a nontrivial infinitesimal deformation of $H^d_2$ different from $N$, where $\phi(z,w) = (w,z)$ and 
\begin{align*}
\phi'(z_1,\ldots, z_{n+1}) = \left\lbrace 
\begin{array}{ll}
  (z_{n+1},\ldots, z_{k+2}, z_{k+1} ,z_{k},\ldots,z_1), &  n = 2 k, k\geq 1\\
  (z_{n+1},\ldots, z_{k+1}, z_{k},\ldots,z_1), &  n = 2 k-1, k \geq 2.
  \end{array}
\right.
\end{align*}
Below  half of all nontrivial infinitesimal deformations of $H^d_2$ are listed as rows of the matrix $Y^d_2$.  The remaining nontrivial elements of $\hol(H^d_2)$ can be deduced by applying $\phi$ and $\phi'$ to each row of $Y^d_2$. 


\begin{align*}
Y^2_2 = \left(\begin{array}{ccc}
a w& - \frac{\bar a z^3}{\sqrt{2}}  & - \bar a z^2 w \\
  - \bar b z^2 w & \frac{b z - \bar b z w^2}{\sqrt{2}}&  0\\
\end{array} \right),
\end{align*}
where $a,b\in \C$, which are also listed in \cite{dSLR17}*{Example 1}. 
\begin{align*}
Y^3_2 = \left(\begin{array}{cccc}
a z - \bar a z^5 & - \frac{2 \bar a z^4 w}{\sqrt{3}} & - \frac{\bar a z^3 w^2}{\sqrt{3}} & 0\\
b w & -\frac{\bar b z^5}{\sqrt{3}} & - \frac{2 \bar b z^4 w}{\sqrt{3}} & -  \bar b z^3 w^2  \\
 c w^2 & 0 & - \frac{\bar c z^4}{\sqrt{3}} &  -  \bar c z^3 w    \\
 d z w & - \frac{\bar d z^4}{\sqrt{3}} &  - \frac{\bar d z^3 w}{\sqrt{3}} &  0 \\
- \bar e z^4 w &  \frac{e z - 2 \bar e z^3 w^2}{\sqrt{3}} & - \frac{\bar e z^2 w^3}{\sqrt{3}} & 0 \\
0  & \frac{f w - \bar f z^4 w}{\sqrt{3}} & - \frac{2 \bar f z^3 w^2}{\sqrt{3}} &  -  \bar f z^2 w^3  \\
0 & \frac{g w^2}{\sqrt{3}} &  - \frac{\bar g z^3 w}{\sqrt{3}} & -  \bar g z^2 w^2 \\
- \bar j z^3 w & \frac{j z^2 - \bar j z^2 w^2}{\sqrt{3}} & 0 & 0\\
0 & \frac{k z w - \bar k z^3 w}{\sqrt{3}} & - \frac{\bar k z^2 w^2}{\sqrt{3}} & 0
\end{array} \right),
\end{align*}
where $a,b,c,d,e,f,g, j, k\in \C$. 
\end{example}

Using the fact that for a sphere map $H$ and any $X\in \hol(H)$ one has $V_HX \in \hol(H^d_n)$ (see \eqref{eq:holHFromHolHd} in the proof of \cref{thm:polyMapInfDef}) it is possible to compute $\hol(H)$ from a description of $\hol(H^d_n)$: By considering each element of $\hol(H^d_n)$ one needs to check if it can be written as $V_H Y$, where $Y$ is a holomorphic vector field, such that $Y \in \hol(H)$. The vector fields obtained in this way span $\hol(H)$. 

\begin{example}
\label{ex:DegMapInfRigid}
The map $H(z,w) = (z,zw,w^2)$ from $\Sphere{3}$ to $\Sphere{5}$ is of finite degeneracy $1$ in $\Sphere{3} \cap \{z=0\}$ and $2$-nondegenerate otherwise. Its infinitesimal stabilizer consists of $S^3_2$. 
It can be checked that $H$ is infinitesimally rigid.
\end{example}

\begin{example}
For the family of sphere maps $G^\ell$ for $\ell \geq 1$ there are infinitesimal rigid maps and some maps which admit nontrivial infinitesimal deformations:
One can compute that for $G^1: \Sphere{3} \rightarrow \Sphere{5}$ the space $\hol(G^1)$ only consists of trivial infinitesimal deformations, see also \cite{DAngelo03}.  
The map $G^4: \Sphere{3} \rightarrow \Sphere{11}$ given by
\begin{align*}
G^4(z,w) = (z^9, c_1 z^7 w,  c_2 z^5 w^2,  c_3 z^3 w^3,  c_4 z w^4, w^9),
\end{align*}
where $c_1 = c_4 = 3, c_2 = 3 \sqrt{3}$ and $c_3 = \sqrt{30}$ is not infinitesimally rigid, since the vector
\begin{align*}
X=\left(0, - \frac{\bar a} {c_1} z^6 w^3, -\frac{6}{c_2} \bar a z^4 w^4, \frac{z^2 w}{c_3}(a z^2 - 9 \bar a w^4), \frac{w^2}{c_4}(3 a z^2 - \bar a w^4), a z w^7\right), \quad a \in \C,
\end{align*}
corresponds to a nontrivial infinitesimal deformation of $G^4$. This can be verified by showing that $X$ is not of the form as the trivial infinitesimal deformation given in \cref{sec:generalInfAuto}. 
\end{example}

Some of the nontrivial infinitesimal deformations of $H^d_n$ originate from curves passing through the map as the following example shows: In the following a family of finitely nondegenerate rational sphere maps is constructed, which contains the homogeneous sphere map $H^d_n$ for $d$ odd. It is well-known that families of sphere maps exist, see the examples in \cite{DAngelo88a} and \cite{FHJZ10}*{Examples 4.1, 4.2}, which motivated the construction. 
See also \cite{DL16} for a study of homotopies of sphere maps.

\begin{theorem}
For $k\geq 1$ the map $H^{2 k + 1}_2: \Sphere{3} \rightarrow \Sphere{2 k + 2}$ is not locally rigid. More precisely, there exists a family $F^{k}_s: \Sphere{3} \rightarrow \Sphere{2k +2}$ of $(2 k + 1)$-nondegenerate rational maps, where $s \in \R$ is sufficiently close to $0$, with $F^k_0 = H^{2 k + 1}_2$ and each $F^k_s$ is not equivalent to $H^{2 k +1}_2$ for $s\neq 0$.
\end{theorem}

\begin{proof}
Consider for $s\in \R$ the map $T_s = (T_s^1,T_s^2): \Sphere{3} \rightarrow \Sphere{3}$ given by
\begin{align*}
T_s(z,w) = \frac{\Bigl(z-1 + (1-s)^2 (z+1),2 (1 - s) w\Bigr)}{2 +s (s - 2) (z+1)},
\end{align*}
such that $T_0(z,w) = (z,w)$, 
which is an automorphism of $\Sphere{3}$ when $s\neq 1$. Define $c_\ell \coloneqq \sqrt{\binom{2 k + 1}{\ell}}$ and note that $c_k = c_{k+1}$. Set
\begin{align*}
F^k_s(z,w) & \coloneqq (c_0 z^{2 k +1},  \ldots,c_{k-1} z^{k+2} w^{k-1}, c_k z^k w^k T_s^1(z,w), c_{k+1} z^k w^k  T_s^2(z,w), \\
 & \qquad c_{k+2} z^{k-1} w^{k+2},\ldots, c_{2k+1} w^{2k+1}).
\end{align*}
It holds for all $s$ that $F^k_s$ maps $\Sphere{3}$ into $\Sphere{2 k + 2}$, since it originates from $H^{2k+1}_2$ by applying the inverse of tensoring to the $(k+1)$-th and $(k+2)$-th component using the map $(z,w) \mapsto (z,w)$, and then tensoring the $(k+1)$-th component of the resulting map with $T_s$. 
Furthermore $F^k_0 = H^{2k+1}_2$ and since $F^k_0$ is $(2k+1)$-nondegenerate by \cref{lem:homMapNonDeg}, the same holds for each map $F^k_s$ for $s$ sufficiently close to $0$.

The remaining step is to show that for $s$ sufficiently close to $0$, when $s \neq 0$ the map $F^k_s$ is not equivalent to a polynomial sphere map (in particular $H^{2k+1}_2$). To see this apply the polynomiality criterion of Faran--Huang--Ji--Zhang \cite{FHJZ10}*{Remark 2.3 (A)}. More precisely, write $F^k_s = \frac P Q$ with $P = (P_1,\ldots, P_{2k+2})$ and $Q$ being the denominator of $T_s$, and $d \coloneqq \deg F^k_s = \max\{\deg P, \deg Q\} = 2k+1$. Consider the map
\begin{align*}
\hat F(z,w,t) & \coloneqq  t^d \left(P\left(\frac z t, \frac w t\right), Q\left(\frac z t, \frac w t\right)\right) \\
& = \Bigl(c_0 z^{2 k +1} (r_1 z + r_2 t),  \ldots,c_{k-1} z^{k+2} w^{k-1} (r_1 z + r_2 t), c_k z^k w^k t(r_2 z + r_1 t), \\
 & \qquad  2 c_{k+1} (1-s)z^k w^{k+1} t, c_{k+2} z^{k-1} w^{k+2} (r_1 z + r_2 t) ,\ldots, c_{2k+1} w^{2k+1}(r_1 z + r_2 t), \\
 & \qquad t^{2k+1} (r_1 z + r_2 t)\Bigr),
\end{align*}
where $r_1 \coloneqq s(s - 2)$ and $r_2 \coloneqq r_1+2$. Write $\hat F = (\hat F_1, \ldots, \hat F_d, \hat Q)$, then $F^k_s$ is equivalent to a polynomial map if and only if there exist $a_1,a_2, A_1, \ldots, A_d \in \C$ and $C \in \C\setminus \{0\}$ such that 
\begin{align}\label{eq:nonpoly}
\sum_{j=1}^d \hat F_j A_j + \hat Q = C (a_1 z + a_2 w + t)^d.
\end{align}
The claim is that under the assumption $s\neq 0$ the equation \eqref{eq:nonpoly} has no solution. Comparing the coefficient of $z^2 t^{d-2}$ one obtains that $a_1 = 0$. Then the coefficient of $z t^{d-1}$ gives $r_1 = 0$, which cannot be satisfied for $0\neq s < 2$. This finishes the proof.
\end{proof}

Note that for $k=1$ the vector $\frac{d}{d s}|_{s=0} F_s^1$ is a nontrivial infinitesimal deformation of $H^3_2$ from \cref{ex:listInfDefHomMap}, when the parameter $k \in \C$  used there is taken to be real.

\bibliography{ref} 

\begin{bibdiv}
\begin{biblist}

\bib{BHR96}{article}{
      author={Baouendi, M.~S.},
      author={Huang, Xiaojun},
      author={{Preiss Rothschild}, Linda},
       title={{Regularity of {CR} mappings between algebraic hypersurfaces}},
        date={1996},
        ISSN={0020-9910},
     journal={Invent. Math.},
      volume={125},
      number={1},
       pages={13\ndash 36},
         url={https://doi.org/10.1007/s002220050067},
      review={\MR{1389959}},
}

\bib{BERbook}{book}{
      author={Baouendi, M.~Salah},
      author={Ebenfelt, Peter},
      author={Rothschild, Linda~Preiss},
       title={{Real submanifolds in complex space and their mappings}},
      series={{Princeton Mathematical Series}},
   publisher={Princeton University Press, Princeton, NJ},
        date={1999},
      volume={47},
        ISBN={0-691-00498-6},
         url={http://dx.doi.org/10.1515/9781400883967},
      review={\MR{1668103}},
}

\bib{DAngelo88b}{article}{
      author={D'Angelo, John~P.},
       title={{Polynomial proper maps between balls}},
        date={1988},
        ISSN={0012-7094},
     journal={Duke Math. J.},
      volume={57},
      number={1},
       pages={211\ndash 219},
  url={http://dx.doi.org.lib-ezproxy.tamu.edu:2048/10.1215/S0012-7094-88-05710-9},
      review={\MR{952233 (89j:32032)}},
}

\bib{DAngelo88a}{article}{
      author={D'Angelo, John~P.},
       title={{Proper holomorphic maps between balls of different dimensions}},
        date={1988},
        ISSN={0026-2285},
     journal={Michigan Math. J.},
      volume={35},
      number={1},
       pages={83\ndash 90},
         url={http://dx.doi.org/10.1307/mmj/1029003683},
      review={\MR{931941 (89g:32038)}},
}

\bib{DAngelo91}{article}{
      author={D'Angelo, John~P.},
       title={{Polynomial proper holomorphic mappings between balls. {II}}},
        date={1991},
        ISSN={0026-2285},
     journal={Michigan Math. J.},
      volume={38},
      number={1},
       pages={53\ndash 65},
         url={http://dx.doi.org/10.1307/mmj/1029004261},
      review={\MR{1091509}},
}

\bib{DAngeloBook}{book}{
      author={D'Angelo, John~P.},
       title={{Several complex variables and the geometry of real
  hypersurfaces}},
      series={{Studies in Advanced Mathematics}},
   publisher={CRC Press, Boca Raton, FL},
        date={1993},
        ISBN={0-8493-8272-6},
      review={\MR{1224231}},
}

\bib{DAngelo03}{article}{
      author={D'Angelo, John~P.},
       title={{Homogenization, reflection, and the {X}-variety}},
        date={2003},
        ISSN={0022-2518},
     journal={Indiana Univ. Math. J.},
      volume={52},
      number={5},
       pages={1113\ndash 1133},
         url={http://dx.doi.org/10.1512/iumj.2003.52.2336},
      review={\MR{2010320}},
}

\bib{DAngelo07a}{article}{
      author={D'Angelo, John~P.},
       title={{The {$X$}-varieties for {CR} mappings between hyperquadrics}},
        date={2007},
        ISSN={1093-6106},
     journal={Asian J. Math.},
      volume={11},
      number={1},
       pages={89\ndash 102},
         url={https://doi.org/10.4310/AJM.2007.v11.n1.a9},
      review={\MR{2304583}},
}

\bib{DAngelo15}{incollection}{
      author={D'Angelo, John~P.},
       title={{C{R} complexity and hyperquadric maps}},
        date={2015},
   booktitle={{Analysis and geometry}},
      series={{Springer Proc. Math. Stat.}},
      volume={127},
   publisher={Springer, Cham},
       pages={17\ndash 34},
         url={https://doi.org/10.1007/978-3-319-17443-3_3},
      review={\MR{3445514}},
}

\bib{DAngelo16}{article}{
      author={D'Angelo, John~P.},
       title={{On the classification of rational sphere maps}},
        date={2016},
        ISSN={0019-2082},
     journal={Illinois J. Math.},
      volume={60},
      number={3-4},
       pages={869\ndash 890},
         url={https://projecteuclid.org/euclid.ijm/1506067297},
      review={\MR{3707643}},
}

\bib{DKR03}{article}{
      author={D'Angelo, John~P.},
      author={Kos, \v{S}imon},
      author={Riehl, Emily},
       title={{A sharp bound for the degree of proper monomial mappings between
  balls}},
        date={2003},
        ISSN={1050-6926},
     journal={J. Geom. Anal.},
      volume={13},
      number={4},
       pages={581\ndash 593},
         url={http://dx.doi.org/10.1007/BF02921879},
      review={\MR{2005154}},
}

\bib{DL16}{article}{
      author={D'Angelo, John~P.},
      author={Lebl, Ji\v{r}{\'i}},
       title={{Homotopy equivalence for proper holomorphic mappings}},
        date={2016},
        ISSN={0001-8708},
     journal={Adv. Math.},
      volume={286},
       pages={160\ndash 180},
         url={http://dx.doi.org/10.1016/j.aim.2015.09.007},
      review={\MR{3415683}},
}

\bib{DX17}{article}{
      author={D'Angelo, John~P.},
      author={Xiao, Ming},
       title={{Symmetries in {CR} complexity theory}},
        date={2017},
        ISSN={0001-8708},
     journal={Advances in Mathematics},
      volume={313},
       pages={590\ndash 627},
  url={http://www.sciencedirect.com/science/article/pii/S0001870817301111},
}

\bib{dSLR15b}{incollection}{
      author={{Della Sala}, Giuseppe},
      author={Lamel, Bernhard},
      author={Reiter, Michael},
       title={{Infinitesimal and local rigidity of mappings of {CR}
  manifolds}},
        date={2017},
   booktitle={{Analysis and geometry in several complex variables}},
      series={{Contemp. Math.}},
      volume={681},
   publisher={Amer. Math. Soc., Providence, RI},
       pages={71\ndash 83},
        note={\href{https://arxiv.org/abs/1710.03963}{arXiv:1710.03963}},
      review={\MR{3603884}},
}

\bib{dSLR15a}{article}{
      author={{Della Sala}, Giuseppe},
      author={Lamel, Bernhard},
      author={Reiter, Michael},
       title={{Local and infinitesimal rigidity of hypersurface embeddings}},
        date={2017},
        ISSN={0002-9947},
     journal={Trans. Amer. Math. Soc.},
      volume={369},
      number={11},
       pages={7829\ndash 7860},
         url={http://dx.doi.org/10.1090/tran/6885},
        note={\href{https://arxiv.org/abs/1507.08842}{arXiv:1507.08842}},
      review={\MR{3695846}},
}

\bib{dSLR18}{article}{
      author={{Della Sala}, Giuseppe},
      author={Lamel, Bernhard},
      author={Reiter, Michael},
       title={{Deformations of CR maps between CR manifolds (submitted)}},
        date={2018},
}

\bib{dSLR17}{article}{
      author={{Della Sala}, Giuseppe},
      author={Lamel, Bernhard},
      author={Reiter, Michael},
       title={{Sufficient and necessary conditions for local rigidity of CR
  mappings and higher order infinitesimal deformations (submitted)}},
        date={2018},
}

\bib{FHJZ10}{article}{
      author={Faran, James},
      author={Huang, Xiaojun},
      author={Ji, Shanyu},
      author={Zhang, Yuan},
       title={{Polynomial and rational maps between balls}},
        date={2010},
        ISSN={1558-8599},
     journal={Pure Appl. Math. Q.},
      volume={6},
      number={3, Special Issue: In honor of Joseph J. Kohn. Part 1},
       pages={829\ndash 842},
         url={http://dx.doi.org/10.4310/PAMQ.2010.v6.n3.a10},
      review={\MR{2677315}},
}

\bib{Forstneric89}{article}{
      author={{Forstneri\v c}, Franc},
       title={{Extending proper holomorphic mappings of positive codimension}},
        date={1989},
        ISSN={0020-9910},
     journal={Invent. Math.},
      volume={95},
      number={1},
       pages={31\ndash 61},
         url={http://dx.doi.org/10.1007/BF01394144},
      review={\MR{969413}},
}

\bib{Lamel01}{article}{
      author={Lamel, Bernhard},
       title={{Holomorphic maps of real submanifolds in complex spaces of
  different dimensions}},
        date={2001},
        ISSN={0030-8730},
     journal={Pacific J. Math.},
      volume={201},
      number={2},
       pages={357\ndash 387},
         url={http://dx.doi.org/10.2140/pjm.2001.201.357},
      review={\MR{MR1875899 (2003e:32066)}},
}

\bib{Lamel11}{incollection}{
      author={Lamel, Bernhard},
       title={{Jet embeddability of local automorphism groups of real-analytic
  {CR} manifolds}},
        date={2011},
   booktitle={{Geometric analysis of several complex variables and related
  topics}},
      series={{Contemp. Math.}},
      volume={550},
   publisher={Amer. Math. Soc., Providence, RI},
       pages={89\ndash 108},
         url={http://dx.doi.org/10.1090/conm/550/10868},
      review={\MR{2868556}},
}

\bib{LM17}{article}{
      author={Lamel, Bernhard},
      author={Mir, Nordine},
       title={{Convergence of formal {CR} mappings into strongly pseudoconvex
  {C}auchy-{R}iemann manifolds}},
        date={2017},
        ISSN={0020-9910},
     journal={Invent. Math.},
      volume={210},
      number={3},
       pages={963\ndash 985},
         url={https://doi.org/10.1007/s00222-017-0743-3},
      review={\MR{3735633}},
}

\bib{Rudin84}{article}{
      author={Rudin, Walter},
       title={{Homogeneous polynomial maps}},
        date={1984},
        ISSN={0019-3577},
     journal={Nederl. Akad. Wetensch. Indag. Math.},
      volume={46},
      number={1},
       pages={55\ndash 61},
      review={\MR{748979}},
}

\bib{Stanton95}{article}{
      author={Stanton, Nancy~K.},
       title={{Infinitesimal {CR} automorphisms of rigid hypersurfaces}},
        date={1995},
        ISSN={0002-9327},
     journal={Amer. J. Math.},
      volume={117},
      number={1},
       pages={141\ndash 167},
         url={http://dx.doi.org.lib-ezproxy.tamu.edu:2048/10.2307/2375039},
      review={\MR{1314461 (96a:32036)}},
}

\end{biblist}
\end{bibdiv}

\end{document}